\documentclass[letter,12pt]{amsart}

\usepackage{comment}

\usepackage{tikz}
\usetikzlibrary{decorations.markings}
\usetikzlibrary{shapes}
\usetikzlibrary{backgrounds}
\usepackage{subfig}
\usepackage{graphicx,caption,float}
\usepackage{amssymb} 

\usepackage{hyperref}
\def\thetitle{{The geometry of purely loxodromic subgroups of right-angled Artin groups}}
\hypersetup{
    pdftitle=   \thetitle,
   pdfauthor=  {Thomas Koberda and Johanna Mangahas and Samuel J. Taylor}
}
\usepackage{enumerate}
\makeatletter
\let\@@enum@org\@@enum@
\def\@@enum@[#1]{\@@enum@org[\normalfont #1]}
\makeatother

\newtheorem{thm}{Theorem}[section]
\newtheorem{lem}[thm]{Lemma}
\newtheorem{cor}[thm]{Corollary}
\newtheorem{prop}[thm]{Proposition}

\newtheorem{que}[thm]{Question}

\theoremstyle{remark}

\theoremstyle{definition}

\newcommand\grp[1]{\langle #1\rangle}

\newcommand\al{\alpha}

\newcommand\R{\mathbb{R}}

\newcommand\Z{\mathbb{Z}}

\newcommand\supp{\operatorname{supp}}

\newcommand\lk{\operatorname{Lk}}
\newcommand\st{\operatorname{St}}

\newcommand\gex{{\Gamma}^e}
\newcommand\aga{{A(\Gamma)}}

\newcommand\gam{\Gamma}
\newcommand\bZ{\mathbb{Z}}

\newcommand\Mod{\operatorname{Mod}}

\newcommand\yt{\widetilde}
\newcommand\salv{S(\Gamma)}
\newcommand\Salv{\yt{S}(\Gamma)}

\begin{document}

\title\thetitle

\author{Thomas Koberda}
\address{Department of Mathematics, University of Virginia, Charlottesville, VA 22904-4137, USA}
\email{thomas.koberda@gmail.com}

\author{Johanna Mangahas}
\address{Department of Mathematics, University at Buffalo, 244 Mathematics Building, Buffalo, NY 14260, USA}
\email{mangahas@buffalo.edu}

\author{Samuel J. Taylor}
\address{Department of Mathematics, Yale University, 10 Hillhouse Ave, New Haven, CT 06520, USA}
\email{s.taylor@yale.edu}

\date{\today}
\begin{abstract}
We prove that finitely generated \emph{purely loxodromic} subgroups of a right-angled Artin group $A(\Gamma)$ fulfill equivalent conditions that parallel characterizations of convex cocompactness in mapping class groups $\Mod(S)$. In particular, such subgroups are quasiconvex in $A(\Gamma)$.  In addition, we identify a milder condition for a finitely generated subgroup of $A(\Gamma)$ that guarantees it is free, undistorted, and retains finite generation when intersected with $A(\Lambda)$ for subgraphs $\Lambda$ of $\Gamma$. These results have applications to both the study of convex cocompactness in $\Mod(S)$ and the way in which certain groups can embed in right-angled Artin groups.
\end{abstract}
\maketitle
\section{Introduction}
\subsection{Overview}
Let $\gam$ be a finite simplicial graph with vertex set $V(\gam)$ and edge set $E(\gam)$, and let $\aga$ be the corresponding right-angled Artin group.  That is, we write \[\aga : = \langle V(\gam)\mid [v_i,v_j]=1 \textrm{ if and only if } \{v_i,v_j\}\in E(\gam)\rangle.\]

Right-angled Artin groups have occupied an important position in geometric group theory in recent years.  Their intrinsic algebraic structure has been of interest since the 1980s~\cite{Droms, Servatius, mihailova}.  Right-angled Artin groups also play a key role in the study of three--manifold topology, culminating in Agol's resolution of the virtual Haken conjecture~\cite{Agol,KahnMarkovic,Wise}.  Right-angled Artin groups are also a prototypical class of CAT(0) groups, and have figured importantly in the study of mapping class groups of surfaces~\cite{CrispWiest, ClayLeiningerMangahas, Koberda, KimKoberda2013,KKIMRN,MangahasTaylor}.

In this article, we concentrate on this lattermost aspect of right-angled Artin group theory, where we think of right-angled Artin groups both as commonly occurring subgroups of mapping class groups of surfaces, and as algebraically similar to mapping class groups themselves.  In particular, we study a class of finitely generated subgroups of right-angled Artin groups, called \emph{purely loxodromic subgroups}.  These are certain finitely generated free subgroups of right-angled Artin groups which we prove are (in a very strong sense) quasiconvex in the ambient right-angled Artin group, and which have quasi-isometric orbit maps to the right-angled Artin group analogue of the curve complex, i.e. the extension graph. 

From this last perspective, we show that purely loxodromic subgroups of right-angled Artin groups are analogous to convex cocompact subgroups of mapping class groups of surfaces.  Convex cocompact subgroups are a class of subgroups of mapping class groups distinguished by Farb and Mosher~\cite{FarbMosher} and which have natural and useful geometric properties. Indeed, our main theorem represents the analogue to an open question for mapping class groups that has received much attention in recent years. (See Section \ref{mcg}  for details.) We shall see that, combining our results with earlier results of the second and third author, purely loxodromic subgroups of right-angled Artin groups generally provide an explicit source of convex cocompact subgroups of mapping class groups.

From our main theorem (Theorem \ref{t:main}), that strong quasiconvexity properties are implied by an element-wise condition on a finitely generated subgroup of a right-angled Artin group, several applications of general interest follow. For example, as referenced above, recent groundbreaking results imply that many geometrically significant groups embed as quasiconvex subgroups of right-angled Artin groups \cite{Agol, HagenWise1}. We show that when such a group contains distorted subgroups (e.g. fiber subgroups of fibered $3$--manifold groups or free-by-cyclic groups) there are elements of these subgroups that map to relatively simple elements of the right-angled Artin group. In particular, the images of these elements have non-cyclic centralizers in $\aga$. See the discussion after Theorem \ref{t:starfree} below for details.

\subsection{Main results}
When a right-angled Artin group $A(\Gamma)$ does not decompose as a direct product, its typical (as in \cite{Sisto}) elements are what we call \emph{loxodromic}.  Equivalently, these are its Morse elements, its elements with contracting axes, its elements that act as rank-one isometries on the CAT(0) space associated to $A(\Gamma)$, and its elements with cyclic centralizers \cite{Servatius, BehrstockCharney,BestvinaFujiwara}.  In \cite{KimKoberda2014}, loxodromic elements are characterized as those with unbounded orbit in the action of $A(\Gamma)$ on its \emph{extension graph} $\Gamma^e$, a hyperbolic space introduced by the first author and Kim  \cite{KimKoberda2013} to study embeddings between right-angled Artin groups.  Here we study \emph{purely loxodromic} subgroups of $A(\Gamma)$: those in which every non-trivial element is loxodromic.  Such subgroups exist in no short supply---in fact, they are typical \cite{TaylorTiozzo}; see Section \ref{s:extensiongraph} for details.  
Our main result is the following:

\begin{thm}\label{t:main}
Suppose that $\Gamma$ is connected and anti-connected and that $H<A(\Gamma)$ is finitely generated.
Then the following are equivalent.
\begin{itemize}
\item[(1)] Some (any) orbit map from $H$ into $\Gamma^e$ is a quasi-isometric embedding.
\item[(2)] $H$ is stable in $A(\Gamma)$.
\item[(3)] $H$ is purely loxodromic.
\end{itemize}
\end{thm}

Here, a graph is called \emph{anti-connected} if its opposite graph is connected. See Subsection \ref{s:notation} for a more detailed discussion. A subgroup $H$ of a finitely generated group $G$ is \emph{stable} if it is undistorted in $G$ and quasi-geodesics in $G$ between points in $H$ have uniformly bounded Hausdorff distance from one another (see Section \ref{s:stability} for details). This property was defined by the third author and Durham in \cite{DurhamTaylor}, where they also show that it implies that $H$ is quasiconvex with respect to any word metric on $G$.  We remark that both stability and condition (1) are strong properties about global geometry---the embedding of the subgroup in either a relevant space or the ambient group---whereas (3) can be read as a purely algebraic condition (that every element in $H$ has cyclic centralizer).  In fact, the proof goes through a fourth equivalence, \emph{join-busting} (we give the definition in Subsection \ref{s:context}), which may be considered a local, combinatorial condition.

A milder condition on $H < A(\Gamma)$ is that none of its nontrivial elements are conjugate into a subgroup generated by a star of $\Gamma$ (see Subsection \ref{s:notation} for complete definitions). 
In that case, we say that $H$ is \emph{star-free}.  For this more general class of subgroups of $A(\Gamma)$, we have:

\begin{thm}\label{t:starfree}
If $\Gamma$ is connected and anti-connected and $H<A(\Gamma)$ is finitely generated and star-free, then:
\begin{itemize}
\item[(1)] $H$ is a free group,
\item[(2)] $H$ is undistorted in $A(\Gamma)$, and
\item[(3)] We have that \[H_{\Lambda}=H \cap A(\Lambda)\] is finitely generated, for any induced subgraph $\Lambda \subset \Gamma$.
\end{itemize}
\end{thm}

The main content of Theorem \ref{t:starfree} is its last two points; statement (1) is known to specialists.  Statement (2) indicates that, among the multiple known constructions of distorted subgroups in right-angled Artin groups, ``star-words'' are a necessary common feature.  In particular, it is a recent theorem of Hagen--Wise that every hyperbolic free-by-cyclic group $G$ has a finite index subgroup $G'$ that quasi-isometrically embeds in a right-angled Artin group $A(\Gamma)$ \cite{HagenWise1, HagenWise2}. Theorem \ref{t:starfree} implies that the fiber subgroup of $G' \le A(\Gamma)$, which is always distorted, necessarily contains words that are conjugate into star subgroups. 

Regarding the third statement of Theorem \ref{t:starfree}, it is not generally true that a finitely generated free subgroup of $A(\Gamma)$ must intersect $A(\Lambda)$ in a finitely generated group; we give a counterexample in Section \ref{s:neumann}. We refer to statement (3) as the \emph{Howson property} for star-free subgroups of right-angled Artin groups.

\subsection{Motivation from mapping class groups}\label{mcg}
In the mapping class group $\Mod(S)$, typical (as in \cite{Rivin, Maher, Sisto}) elements are \emph{pseudo-Anosov}.  Here pseudo-Anosov mapping classes may be defined as those with unbounded orbit in the mapping class group action on its associated \emph{curve graph} $\mathcal{C}(S)$, a famously hyperbolic space \cite{MasurMinsky}.  Pseudo-Anosov elements are also alternately characterized by being Morse, having virtually infinite cyclic centralizers, and having contracting axes in relevant spaces \cite{Behrstock, Minsky96}.

Farb and Mosher \cite{FarbMosher} introduced the notion of \emph{convex cocompact} subgroups of $\Mod(S)$, which they defined as the finitely generated subgroups with quasiconvex orbit in Teichm\"{u}ller space.  Such subgroups are \emph{purely pseudo-Anosov}, meaning all infinite-order elements are pseudo-Anosov.  Subsequent work has led to alternate characterizations: Kent and Leininger, and, independently, Hamenst\"{a}dt, proved the equivalence of convex cocompactness and (1) below, while Durham and Taylor proved the equivalence of convex cocompactness and (2).

\begin{thm}[\cite{DurhamTaylor,Hamenstadt,KentLeininger}] \label{t:cc}
A finitely generated $H<\Mod(S)$ is convex cocompact if and only if either of the following equivalent conditions hold:
\begin{itemize}
\item[(1)] Some (any) orbit map from $H$ into $\mathcal{C}(S)$ is a quasi-isometric embedding.
\item[(2)] $H$ is stable in $\Mod(S)$.
\end{itemize}
\end{thm}

Postponing to Section \ref{s:context} our comments on the apparent similarities between Theorems \ref{t:main} and \ref{t:cc}, let us first remark that Theorem \ref{t:cc} does not include the element-wise condition known to be necessary for convex cocompactness: that the group $H$ is purely pseudo-Anosov. In fact, this is an open question that has received much recent attention.

\begin{que}[Farb--Mosher]\label{q:ppA}Are finitely generated, purely pseudo-Anosov subgroups necessarily convex cocompact?\end{que}

So far, this question has been answered only in special cases.  It is easily seen to be true for subgroups of Veech groups, which preserve a hyperbolic disk isometrically embedded in Teichm\"{u}ller space \cite{KentLeininger2}.  A more significant case is resolved by \cite{DowdallKentLeininger} (generalizing \cite{KentLeiningerSchleimer}) who answer Question \ref{q:ppA} affirmatively for subgroups of certain hyperbolic 3-manifold groups embedded in $\Mod(S)$.

Theorem \ref{t:main} completes the affirmative answer for a third case, the family of mapping class subgroups studied by the second and third authors in \cite{MangahasTaylor}.  That paper considers what the authors call \emph{admissible embeddings} of $A(\Gamma)$ in $\Mod(S)$, whose abundant existence is established in \cite{ClayLeiningerMangahas}.  The central result of \cite{MangahasTaylor} is that, if $A(\Gamma) < \Mod(S)$ is admissible, then a finitely generated $H < A(\Gamma) < \Mod(S)$ is convex cocompact if and only if it is purely pseudo-Anosov and quasiconvex in $A(\Gamma)$ using the standard word metric.  Since for any embedding of $A(\Gamma)$ into $\Mod(S)$, pseudo-Anosov elements in the target are the images of loxodromic elements of $\aga$, Theorem \ref{t:main} renders the quasiconvex condition redundant:

\begin{cor}\label{c:answerquestion}If $A(\Gamma) < \Mod(S)$ is an admissible embedding, then any finitely generated $H < A(\Gamma) < \Mod(S)$ is convex cocompact if and only if it is purely pseudo-Anosov.
\end{cor}

\subsection{Motivation from right-angled Artin group geometry}\label{s:context}
From another point of view, Theorem \ref{t:main} can be read as an affirmative answer for the right-angled Artin group analogue of Question \ref{q:ppA}. More precisely, Question \ref{q:ppA} asks whether for a finitely generated subgroup $H\le \Mod(S)$, the element-wise condition of being pseudo-Anosov is strong enough to guarantee the global property of being convex cocompact. Theorem \ref{t:main} gives a positive answer to the corresponding question within a right-angled Artin group. The analogy arises from parallel conditions $(1)$ and $(2)$ in Theorems \ref{t:main} and \ref{t:cc}, along with the corresponding characterizations of loxodromic elements in $A(\Gamma)$ and pseudo-Anosov elements in $\Mod(S)$. We refer to \cite{KimKoberda2014} for a detailed account of the analogy between $\Gamma^e$ and $\mathcal{C}(S)$.

The proof of Theorem \ref{t:main} establishes the equivalence of (1)-(3) via a fourth condition, that $H$ is \emph{join-busting}.  Let $\gam$ be a connected and anti-connected graph, and let $H<\aga$ be a finitely generated purely loxodromic subgroup.  For a natural number $N$, we say that $H$ is \emph{$N$--join-busting} if whenever $w\in H$ is a reduced word in $\aga$ and $\beta\leq w$ is a join subword, then $\ell_{\aga}(\beta)\leq N$. (Join subwords are defined in Section \ref{s:notation}; they are the portions of the word $w$ that live in flats of $A(\Gamma)$.)  We say $H$ is join-busting if it is $N$--join-busting for some $N$.  Intuitively, geodesics between elements in a join-busting subgroup can spend only bounded bursts of time in embedded product regions.

\subsection{Organization of paper} 
Section \ref{s:background} fixes notation and background about right-angled Artin groups, including in Subsection \ref{s:extensiongraph} a summary of results about extension graphs and recipes for producing many purely loxodromic subgroups.

An important tool for the proof of Theorem \ref{t:main} are what we call disk diagrams, which are dual to van Kampen diagrams and relate to the cubical structure of $A(\Gamma)$.  These are defined in Section \ref{s:dds} and yield useful lemmas collected in Section \ref{s:ddapps}.

In Section \ref{s:jbusting}, we prove Theorem \ref{t:jbusting}, that finitely generated purely loxodromic subgroups are join-busting.  In Section \ref{s:stability}, we prove that join-busting subgroups fulfill Statement (2) of Theorem \ref{t:main} (Corollary \ref{cor:joinbusting_implies_stable}), and we give more details on that condition, called \emph{stability}, which is a kind of strong quasiconvexity.

We prove Theorem \ref{t:main} in Section \ref{s:cnvxccmpct}, and Theorem \ref{t:starfree} in Section \ref{s:neumann}.

\subsection{Acknowledgments}
The authors thank S. Dowdall and I. Kapovich for helpful conversations. The authors are grateful to an anonymous referee for a very detailed report which greatly improved the paper. The first named author was partially supported by NSF grant DMS-1203964 while this work was completed.  The second named author is partially supported by NSF DMS-1204592 and was a participant at ICERM while this work was completed.  The third named author is partially supported by NSF grant DMS-1400498. 

\section{Background}\label{s:background}

\subsection{Coarse geometry}\label{s:coarse_geo}
For metric spaces $(X,d_X)$ and $(Y,d_Y)$ and constants $K\ge 1$ and $L \ge 0$, a map $f: X \to Y$ is a \emph{$(K,L)$-quasi-isometric embedding} if for all $x_1,x_2 \in X$,
\[\frac{1}{K} d_X(x_1,x_2) -L \le d_Y(f(x_1),f(x_2)) \le K d_X(x_1,x_2) +L.\]
A \emph{quasi-isometric embedding} is simply a $(K,L)$--quasi-isometric embedding for some $K,L$.  If the inclusion map of a finitely generated subgroup $H$ into the finitely generated group $G$ is a quasi-isometric embedding, we say $H$ is \emph{undistorted} in $G$ (this is independent of the word metrics on $H$ and $G$).

When a quasi-isometric embedding $f:X\to Y$ has the additional property that every point in $Y$ is within a bounded distance from the image $f(X)$, we say $f$ is a \emph{quasi-isometry} and $X$ and $Y$ are \emph{quasi-isometric}.

Where $I$ is a subinterval of $\R$ or $\Z$, we call a $(K,L)$-quasi-isometric embedding $f:I \to Y$ a \emph{$(K,L)$-quasi-geodesic}. If $K =1$ and $L=0$, then $f: I \to Y$ is a \emph{geodesic}. When we refer to a $K$-quasi-geodesic, we mean a $(K,K)$-quasi-geodesic as we have defined it here.

A subset $C \subset X$ is \emph{$K$-quasiconvex} if for any $x,y \in C$ and any geodesic $[x,y]$ between $x,y$ in $X$,
\[[x,y] \subset N_{K}(C),\]  
where $N_K(\cdot)$ denotes the $K$-neighborhood. We say $C$ is \emph{quasiconvex} if it is $K$-quasiconvex for some $K$.  When we speak of a quasiconvex subgroup $H$ of a group $G$, we have fixed a word metric on $G$ with respect to some finite generating set (changing generating sets can change which subgroups are quasiconvex). 

A quasi-geodesic $\gamma$ in a metric space $X$ is called \emph{Morse} or \emph{stable }if every $(K,L)$--quasi-geodesic with endpoints on $\gamma$ is within a bounded distance of $\gamma$, depending only on $K$ and $L$. For a finitely generated group $G$ and $g \in G$, $g$ is called a \emph{Morse element} of $G$ if $\langle g \rangle$ is a Morse quasi-geodesic in $G$ with respect to some (any) finite generating set for $G$. A generalization of stability to subgroups $H \le G$ is recalled in Section \ref{s:stability}.

If $X$ is a geodesic metric space, i.e. every two points are joined by some geodesic, then $X$ is said to be $\delta$-hyperbolic if for any $x,y,z$ and geodesic segments $[x,y],[y,z],[x,z]$ joining them, we have $[x,y] \subset N_{\delta}([x,y] \cup [y,z])$. Here, $N_{\delta}(\cdot)$ denotes a $\delta$-neighborhood in $X$. The space $X$ is \emph{hyperbolic} if $X$ is $\delta$-hyperbolic for some $\delta \ge 0$. We will require some standard terminology about group actions on hyperbolic spaces, so suppose that $G$ acts on $X$ by isometries. An element $g \in G$ acts \emph{elliptically} on $X$ if $g$ has a bounded diameter orbit in the space $X$. At the other extreme, $g\in G$ acts \emph{loxodromically} on $X$ if for $x \in X$ the stable translation length 
\[
 \liminf_{n\to \infty} \frac{d_X(x,g^nx)}{n}
\]
is strictly positive. In this case, it follows easily that any orbit of $\langle g \rangle$ in $X$ is a quasi-geodesic. Finally, the action $G \curvearrowright X$ is said to be \emph{acylindrical} if for every $r >0$ there are $N,R >0$ such that for any $x,y\in X$ with $d(x,y) \ge R$ there are no more than $N$ elements $g \in G$ with 
\[
d(x,gx)\le r \quad \text{and} \quad d(y,gy)\le r .
 \]
For additional facts about groups acting acylindrically on hyperbolic spaces, see \cite{osin2015acylindrically}.

\subsection{RAAGs, graphs, joins, and stars}\label{s:notation} 
Fix a finite, simplicial graph $\Gamma$ with vertex set $V(\Gamma)$ and edge set $E(\Gamma)$, where edges are unordered pairs of distinct vertices.  The \emph{right-angled Artin group} with defining graph $\Gamma$ is the group $A(\Gamma)$ presented by 
\[ \langle v \in V(\Gamma) : [v,w] = 1 \text{ if and only if } \{v,w\} \in E(\Gamma) \rangle.\]
A graph $\gam$ is a \emph{join} if there are nonempty subgraphs $J_1,J_2\subset\gam$ such that $\gam=J_1*J_2$.  Here, the star notation means that every vertex of $J_1$ is adjacent to every vertex of $J_2$. We remark that $A(\Gamma)$ decomposes as a direct product if and only if $\Gamma$ is a join.

We call the generators $v \in V(\Gamma)$ the \emph{vertex generators} of $A(\Gamma)$ and note that whenever we say that an element $g\in A(\Gamma)$ is represented by a word $w$, we always mean a word in the vertex generators. Call a word $w$ representing $g \in A(\Gamma)$ \emph{reduced} if it has minimal length among all words representing $g$. Throughout, we use the symbol ``$=$''  
to denote equality in the group $A(\Gamma)$, i.e. $w=w'$ if and only if $w$ and $w'$ represent the same element in $A(\Gamma)$. To say that $w$ and $w'$ are equal as words in the vertex generators, we will use the notation $w \equiv w'$.

Let $w$ be any word in the vertex generators. We say that $v\in V(\gam)$ is in the \emph{support} of $w$, written $v\in \supp(w)$, if either $v$ or $v^{-1}$ occurs as a letter in $w$. For $g \in\aga$ and $w$ a reduced word representing $g$, we define the support of $g$, $\supp(g)$, to be $\supp(w)$. It follows from \cite{HermillerMeier} that $\supp(g)$ is well-defined, independent of the reduced word $w$ representing $g$. Indeed, this fact also follows from the techniques used in Section \ref{s:dds}.

A simplicial graph is \emph{connected} if it is connected as a simplicial complex.  The \emph{opposite graph} of a simplicial graph $\gam$ is the graph $\gam^{opp}$, whose vertices are the same as those of $\gam$, and whose edges are $\{\{v,w\}\mid \{v,w\}\notin E(\gam)\}$.  A simplicial graph is \emph{anti-connected} if its opposite graph is connected.  Note that a graph is anti-connected if and only if it does not decompose as a nontrivial join.

Throughout this work, by a subgraph $\Lambda \subset \Gamma$ we will always mean a subgraph which is induced from $\Gamma$. Recall that a subgraph is \emph{induced} (or \emph{full}) if any edge of $\Gamma$ whose vertices are contained in $\Lambda$ is also an edge of $\Lambda$. For such a subgraph, the natural homomorphism $A(\Lambda) \to \aga$ is injective and, in fact, is an isometric embedding with respect to the word metrics given by the vertex generators. This last fact is not special to right-angled Artin groups, but in fact holds for general Artin groups~\cite{CP2014}. We identify $A(\Lambda)$ with its image in $\aga$ and refer to it as the subgroup of $\aga$ induced by $\Lambda \subset \Gamma$. Hence, for any $g \in \aga$ and subgraph $\Lambda \subset \Gamma$, $\supp(g) \subset \Lambda$ if and only if $g \in A(\Lambda)$.

Of particular importance are the subgroups of $A(\Gamma)$ that nontrivially decompose as direct products. Say that an induced subgraph $J$ of $\Gamma$ is a \emph{join} of $\Gamma$ if $J$ itself is a join graph and call any subgraph $\Lambda \subset J$ of a join graph a \emph{subjoin}; note that this definition does not require a subjoin be itself a join. The induced subgroup $A(J) \le A(\Gamma)$ a \emph{join subgroup}.  A word $w$ in $\aga$ is called a $\emph{join word}$ if $\supp(w) \subset J$ for some join $J$ of $\Gamma$. Hence, a reduced word $w$ is a join word if and only if $w$ represents $g \in \aga$ such that $g \in A(J)$. If $\beta$ is a subword of $w$, we will say that $\beta$ is a \emph{join subword} of $w$ when $\beta$ is itself a join word. If an element $g \in \aga$ is conjugate into a join subgroup, then we will call $g$ \emph{elliptic}. This terminology is justified by the discussion in Section \ref{s:extensiongraph}. It is a theorem of Servatius \cite{Servatius} that $g \in A(\Gamma)$ lives in a $\Z \oplus \Z$ subgroup of $A(\Gamma)$ if and only if $g$ is an elliptic element of $A(\Gamma)$. (This is part of Theorem \ref{thm:loxodromics} below).

Let $v\in V(\gam)$ be a vertex.  The \emph{star} of $v$, written $\st(v)$, is the vertex $v$ together with the \emph{link} of $v$, written $\lk(v)$, which is the set of vertices adjacent to $v$.  A \emph{star subgroup} is a subgroup of a right-angled Artin group generated by the vertices in a star, i.e. $A(\st(v))$ for some vertex $v$.  Note that a star of a vertex is always a join, but the converse is generally not true.  A reduced word $w\in\aga$ is a \emph{star word} if it is supported in a star of $\Gamma$. A word $w$ (or subgroup $H$) of $A(\Gamma)$ is \emph{conjugate into a star} if some conjugate of $w$ (resp. $H$) is contained in a star subgroup of $A(\Gamma)$. Note again that a star word is always a join word, but the converse is generally not true.

\subsection{The Salvetti complex and its hyperplanes}\label{s:salvetti}
In the study of the geometry of a right-angled Artin group, it is often useful to consider a canonical $K(G,1)$ for a right-angled Artin group $\aga$, called the \emph{Salvetti complex} of $\aga$, which we denote by $\salv$.  For the convenience of the reader, we recall a construction of $\salv$.

Let $n=|V(\gam)|$.  We consider the $1$--skeleton of the unit cube $I^n\subset\R^n$ with the usual cell structure, and we label the $1$--cells of $I^n$ which contain the origin, by the vertices of $\gam$.  If $K$ is a complete subgraph of $\gam$, we attach the corresponding sub-cube of $I^n$ labeled by the relevant $1$--cells.  Doing this for every complete subgraph of $\gam$, we obtain a subcomplex $X(\gam)\subset I^n$.  We write $\salv=X(\gam)/\Z^n$, where we have taken the quotient by the usual $\Z^n$--action on $\R^n$.

It is easy to check that $\salv$ is a compact cell complex and that $\pi_1(\salv)\cong\aga$.  By the Milnor--Schwarz Lemma, we have that $\aga$ is quasi-isometric to $\Salv$, the universal cover of $\salv$.  By construction, the labeled $1$--skeleton of $\Salv$ is also the Cayley graph of $A(\Gamma)$ with respect to the generating set $V(\Gamma)$.  Consequently, reduced words in $A(\Gamma)$ correspond to geodesics in $\Salv^{(1)}$, which we call \emph{combinatorial geodesics}.  We refer to distance in $\Salv^{(1)}$ as \emph{combinatorial distance}.  By \emph{combinatorial quasiconvexity}, we mean quasiconvexity in $\Salv^{(1)}$.

It is well-known that $\Salv$ is a CAT(0) cube complex.  This structure will be important to our proofs, so we record some relevant definitions here.  In $\salv$, each edge $e_v$ (labeled by the generator $v$) is dual to a \emph{hyperplane} $H_v$ determined as follows: let $m$ be the midpoint of $e_v$.  For every sub-cube $e_v \times I^k$ in $X(\gam)$, the \emph{midcube dual to} $e_v$ is the set $m \times I^k$; we define $H_v$ as the image in $\salv$ of all midcubes dual to $e_v$.  Observe that, by construction, $H_v$ and $H_u$ intersect if and only if $v$ and $u$ commute.  The lifts of hyperplanes in $\salv$ define similarly labeled hyperplanes dual to edges in $\Salv$.  Each of these latter hyperplanes separates $\Salv$ into two convex sets.  It follows that the combinatorial distance between a pair of vertices equals the number of hyperplanes in $\Salv$ separating those vertices. These facts hold generally for CAT(0) cube complexes: see Lecture 1 of~\cite{sageevnotes} or Proposition 3.4 of~\cite{schwernotes}.

\subsection{Extension graphs and a profusion of purely loxodromic subgroups}\label{s:extensiongraph}
In this section, we record some facts about the extension graph $\gex$, and we establish the existence of many purely loxodromic subgroups of right-angled Artin groups. Recall that $g \in \aga$ is said to be elliptic if it is conjugate into a join subgroup of $\aga$. Otherwise, we say that $g \in \aga$ is \emph{loxodromic}. Also, call $g\in \aga$ \emph{cyclically reduced} if it has the shortest word length among all of its conjugates in $\aga$.

Let $\gam$ be a finite simplicial graph.  The \emph{extension graph} $\gex$ of $\gam$ is the graph whose vertices are given by $\{v^g\mid v\in V(\Gamma),g\in\aga\}$, and whose edges are given by pairs $\{\{v^g,w^h\}\mid[v^g,w^h]=1\}$. Here,  $v^g =g^{-1}vg$.

Note that the conjugation action of $\aga$ on itself induces an action of $\aga$ on $\gex$ by isometric graph automorphisms.

\begin{thm}[See \cite{KimKoberda2013}, \cite{KimKoberda2014}]\label{t:gex}
Let $\gam$ be a finite, connected, and anti-connected simplicial graph with at least two vertices.
\begin{enumerate}
\item
The graph $\gex$ is connected.
\item
The graph $\gex$ has infinite diameter.
\item
The graph $\gex$ is locally infinite.
\item
The graph $\gex$ is quasi-isometric to a regular tree of infinite valence.  In particular, $\gex$ is $\delta$--hyperbolic.
\item
The action of $\aga$ on $\gex$ is acylindrical.  In particular, the action of every nonidentity element of $\aga$ is either elliptic or loxodromic, depending on whether its orbits in $\gex$ are bounded or unbounded.
\item
A nontrivial cyclically reduced element of $\aga$ acts elliptically on $\gex$ if and only if its support is contained in a join of $\gam$.  In particular, $\aga$ contains a loxodromic element.
\item
Any purely loxodromic subgroup of $\aga$ is free.
\end{enumerate}
\end{thm}

Hence, using Theorem \ref{t:gex}, we see that when $\Gamma$ if connected and anti-connected $g \in A(\Gamma)$ is \emph{loxodromic} if and only if it acts as a loxodromic isometry of $\Gamma^e$. Combining results of \cite{Servatius, BehrstockCharney,KimKoberda2014} we state the following characterization of loxodromic elements of $A(\Gamma)$ for easy reference to the reader.

\begin{thm}[Characterization of loxodromics] \label{thm:loxodromics}
Let $g\in A(\Gamma)$ for $\Gamma$ connected, anti-connected, and with at least two vertices. The following are equivalent:
\begin{enumerate}
\item $g$ is loxodromic, i.e. $g$ is not conjugate into a join subgroup.
\item $g$ acts as a rank $1$ isometry on $\Salv$.
\item The centralizer $C_{\aga}(g)$ is cyclic.
\item $g$ acts as a loxodromic isometry of $\gex$.
\end{enumerate}
\end{thm}

A single loxodromic element of $\aga$ generates an infinite, purely loxodromic subgroup of $\aga$.  Items (4) and (5) of Theorem \ref{t:gex} allow us to produce many non-cyclic purely loxodromic subgroups:

\begin{prop}\label{p:dahguiosin}
Let $\lambda\in\aga$ be a loxodromic element for $\Gamma$ connected and anti-connected.  There is an $N$ which depends only on $\gam$ such that the normal closure $\Lambda_N$ of $\lambda^N$ in $\aga$ is purely loxodromic.
\end{prop}
\begin{proof}
We have that $\gex$ is a $\delta$--hyperbolic graph on which $\aga$ acts acylindrically.  By \cite{dahguiosin}, Theorem 2.30, we have $N\gg 0$ for which the normal closure $\Lambda_N$ of $\lambda^N$ is purely loxodromic and free, and this constant $N$ depends only on the hyperbolicity and acylindricity constants of the $\aga$ action on $\gex$.
\end{proof}

We remark that, if a finite set of loxodromic elements of $A(\Gamma)$ is pairwise independent (i.e., no two elements have common powers), then a standard argument in hyperbolic geometry proves that sufficiently high powers of these elements generate a purely loxodromic group.  Moreover, by following the arguments in \cite{dahguiosin}, one can verify a generalization of Propositon \ref{p:dahguiosin}: there exist powers of these elements whose normal closure (as a set) is purely loxodromic.

Proposition \ref{p:dahguiosin} shows that there are enough loxodromic elements in $\aga$ for them to be present in every nontrivial normal subgroup:

\begin{thm}\label{t:loxsubgroup}
Let $\gam$ be a connected and anti-connected graph with at least two vertices, and let $1\neq G<\aga$ be a normal subgroup.  Then $G$ contains a loxodromic element of $\aga$.
\end{thm}
It is very easy to produce counterexamples to Theorem \ref{t:loxsubgroup} if one omits the hypothesis that $\gam$ be anti-connected.
\begin{proof}[Proof of Theorem \ref{t:loxsubgroup}]
Let $1\neq g\in G$ be any element, and let $\lambda\in\aga$ be loxodromic.  Let $N$ be the constant furnished by Proposition \ref{p:dahguiosin}.  Consider the element $[g,\lambda^N]\in G\cap\Lambda_N$.  If this commutator is the identity then $g$ commutes with $\lambda^N$, so that these two elements have a common power.  In particular, $g$ is loxodromic.  If the commutator is not the identity then it represents a nontrivial loxodromic element of $G$, since $\Lambda_N$ is purely loxodromic.
\end{proof}

We remark that Behrstock, Hagen, and Sisto~\cite{behrhagsis} have very recently introduced a different curve complex analogue, called the contact graph, associated to RAAGs and more generally CAT(0) cubical groups.  Such a group's action on its contact graph may be used to define ``loxodromic" elements in the group.  In the RAAG case, the contact graph $\mathcal{C}\Salv$ is equivariantly quasi-isometric to the extension graph $\Gamma^e$, except for in a few sporadic cases.  Therefore both definitions of loxodromic elements coincide, and one may add to Theorem \ref{t:main} a statement (1'), replacing $\Gamma^e$ with $\mathcal{C}\Salv$.

Finally, we note the work of the third author and Tiozzo~\cite{TaylorTiozzo}, which proves that typical subgroups of appropriate $A(\Gamma)$ are purely loxodromic.  More precisely, let $\mu$ be a probability measure on $\aga$ whose support generates $\aga$ as a semigroup, and consider $k$ independent random walks on $\aga$ whose increments are distributed according to $\mu$.  They prove that the subgroup generated by the $n^{th}$ step of each random walk is purely loxodromic with probability going to $1$ as $n$ goes to infinity.

\section{Disk diagrams in RAAGs}\label{s:dds}
Our main technical tool is the disk diagrams introduced in \cite{CrispWiest} (as ``dissections of a disk'') and further detailed in \cite{Kim}.  We present the general theory here for the convenience of the reader.

\subsection{Construction of a disk diagram}\label{s:construct}
Suppose that $w$ is a word in the vertex generators of $A(\Gamma)$ and their inverses which represents the identity element in $A(\Gamma)$.  A disk diagram for $w$ is a combinatorial tool indicating how to reduce $w$ to the empty word.  Before giving a formal definition, we show how to construct such a diagram.

\begin{figure}[H]
\begin{center}
\includegraphics[height=2in]{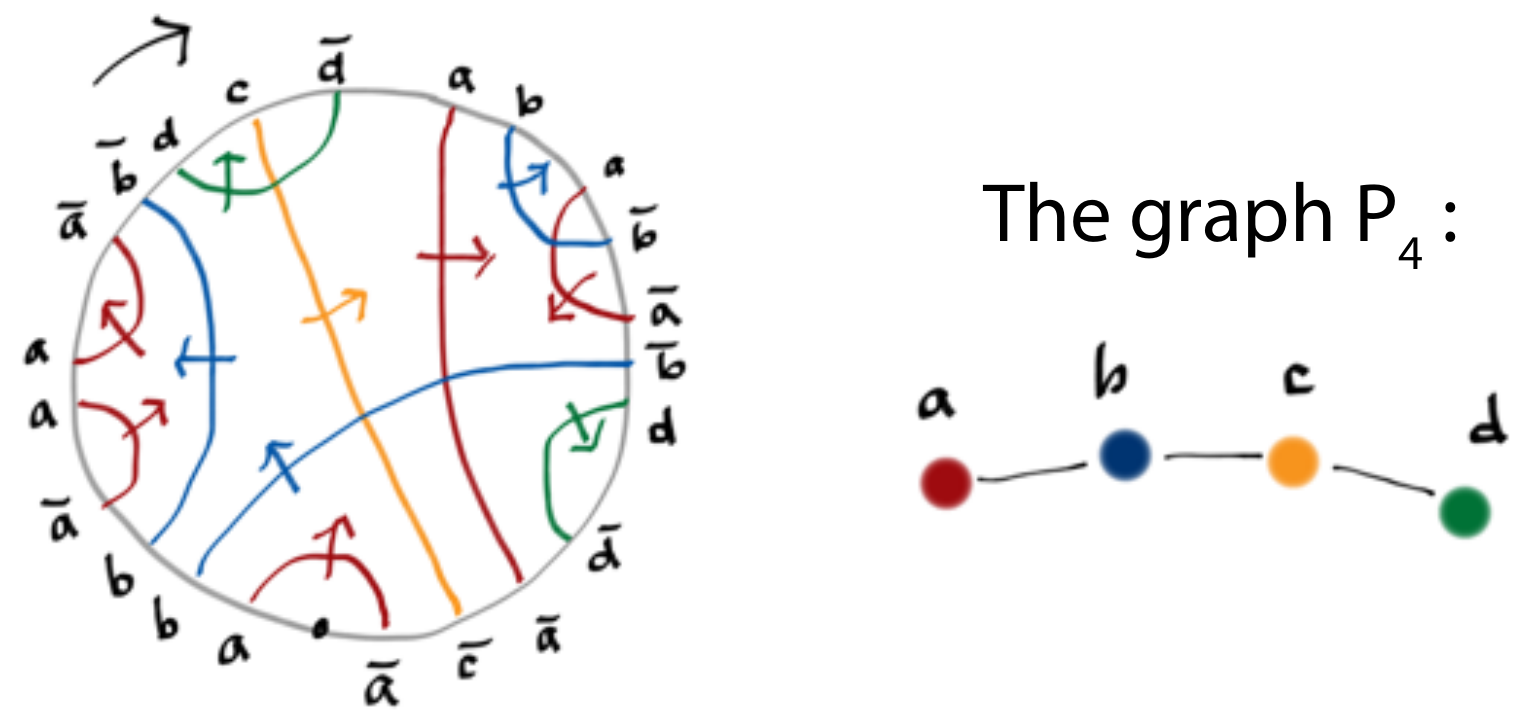}
\caption{On the left, a disk diagram for the identity word $abb\bar{a}aa\bar{a}\bar{b}dc\bar{d}aba\bar{b}\bar{a}\bar{b}d\bar{d}\bar{a}\bar{c}\bar{a}$ in $A(P_4)$, with $P_4$ shown on the right.  Here $\bar{x}$ is the inverse of $x$.}
\label{dd1}
\end{center}
\end{figure}

Consider the loop in $\salv$ determined by the word $w$, which we continue to denote by $w$. This is the map $w: S^1 \to \salv$ that, starting at a basepoint in $S^1$, parameterizes an edge path in $\salv$ spelling the word $w$. (Throughout this paper, an edge path in $\salv$, or its universal cover, \emph{spells} a word $w$ if $w$ is the word obtained by recording the labels on the oriented edges traversed by the path according to its orientation.) Since $w$ represents the trivial word in $A(\Gamma)$, this map is null homotopic in $\salv$. Denote by $f: D \to \salv$ an extension of the loop $w$ to the unit disk, and homotope $f$ rel $\partial D$ so that it is transverse to the hyperplanes  of $\salv$. After doing so, all preimages of hyperplanes through $f$ consist of simple closed curves and properly embedded arcs. Using standard techniques in combinatorial topology, perform a further homotopy rel $\partial D$ to remove all simple closed curve components in the preimage and to ensure that no properly embedded arcs intersect more than once. After performing such a homotopy, the preimage of each hyperplane in $\salv$ is a collection of properly embedded disjoint arcs. 

For the hyperplane $H_v$ in $\salv$ dual to the edge $e_v$ representing $v$, there are two co-orientations corresponding to the two orientations of the edge $e_v$. We label each of these two co-orientations on $H_v$ either $v$ or $v^{-1}$ corresponding to the label on the oriented edge. By pulling back the co-orientation of $H_v$, we likewise label the co-orientation of each arc in the preimage of $H_v$. This is to say that an arc in the preimage of $H_v$ is labeled with both the symbol $v$ and a transverse arrow. Crossing the arc in the direction indicated by the arrow corresponds to reading the label $v$ and crossing in the reverse direction reads the label $v^{-1}$. The oriented disk $D$ together with this collection of labeled, co-oriented properly embedded arcs is called a disk diagram for $w$ and is denoted by $\Delta$. See Figure \ref{dd1}.

\subsection{Formal definition}\label{s:define}
Our disk diagrams follow the definition in \cite{Kim}. For a word $w$ representing the trivial element in $A(\Gamma)$, a \emph{disk diagram} (or dual van Kampen diagram) $\Delta$ for $w$ in $A(\Gamma)$ is an oriented disk $D$ together with a collection $\mathcal{A}$ of co-oriented, properly embedded arcs in general position, satisfying the following:
\begin{enumerate}
\item Each co-oriented arc of $\mathcal{A}$ is labeled by an element of $V(\Gamma)$.  Moreover, if two arcs of $\mathcal{A}$ intersect then the generators corresponding to their labels are adjacent in $\Gamma$.
\item With its induced orientation, $\partial D$ represents a cyclic conjugate of the word $w$ in the following manner: there is a point $* \in \partial D$ such that $w$ is spelled by starting at $*$, traversing $\partial D$ according to its orientation, and recording the labels of the arcs of $\mathcal{A}$ so encountered, \emph{according to their co-orientation}.  The latter means that, at each intersection of an arc of $\mathcal{A}$ with $\partial D$, the corresponding element of $V(\Gamma)$ is recorded if the co-orientation of the arc coincides with the orientation of $\partial D$, and its inverse is recorded otherwise.
\end{enumerate} 

We think of the boundary of $D$ as subdivided into edges and labeled according to the word $w$, with edges oriented consistently with $\partial D$ if that edge is labeled by an element of $V(\Gamma)$, and oppositely oriented otherwise, so that the label on an oriented edge in this subdivision agrees with the label on the co-oriented arc of $\mathcal{A}$ meeting this edge. In this way, each arc of $\mathcal{A}$ corresponds to two letters of $w$ which are represented by oriented edges on the boundary of $\Delta$.


While not required by the definition, it is convenient to restrict our attention to \emph{tight} disk diagrams, in which arcs of $\mathcal{A}$ intersect at most once.

\subsection{Interchangeable interpretations}
In Section \ref{s:construct}, we saw that any word representing the identity has a disk diagram induced by a null-homotopy $f:D\to\salv$.  Here we show how any disk diagram can be induced by such a map.  This uses the duality, observed in \cite{CrispWiest} and detailed in \cite{Kim}, between disk diagrams and the more familiar notion of van Kampen diagrams.  For a general discussion of the latter, see \cite{Sapir} or \cite{LyndonSchupp}.  We recall the relevant details here.

A \emph{van Kampen diagram} $X \subset S^2$ for a word $w$ representing the identity in $A(\Gamma)$ is a simply connected planar $2$--complex equipped with a combinatorial map into $\salv$ whose boundary word is $w$. Such a map may be expressed by labeling the oriented $1$--cells of the diagram $X$ with vertex generators and their inverses in such a way that every $2$--cell represents a (cyclic conjugate of) a relator of $A(\Gamma)$, i.e. a $2$--cell of $\salv$.  The boundary word of $X$ is the cyclic word obtained by reading the labels of $\partial X$, and $X$ represents a null-homotopy of its boundary.

\begin{figure}[H]
\begin{center}
\includegraphics[height=1.75in]{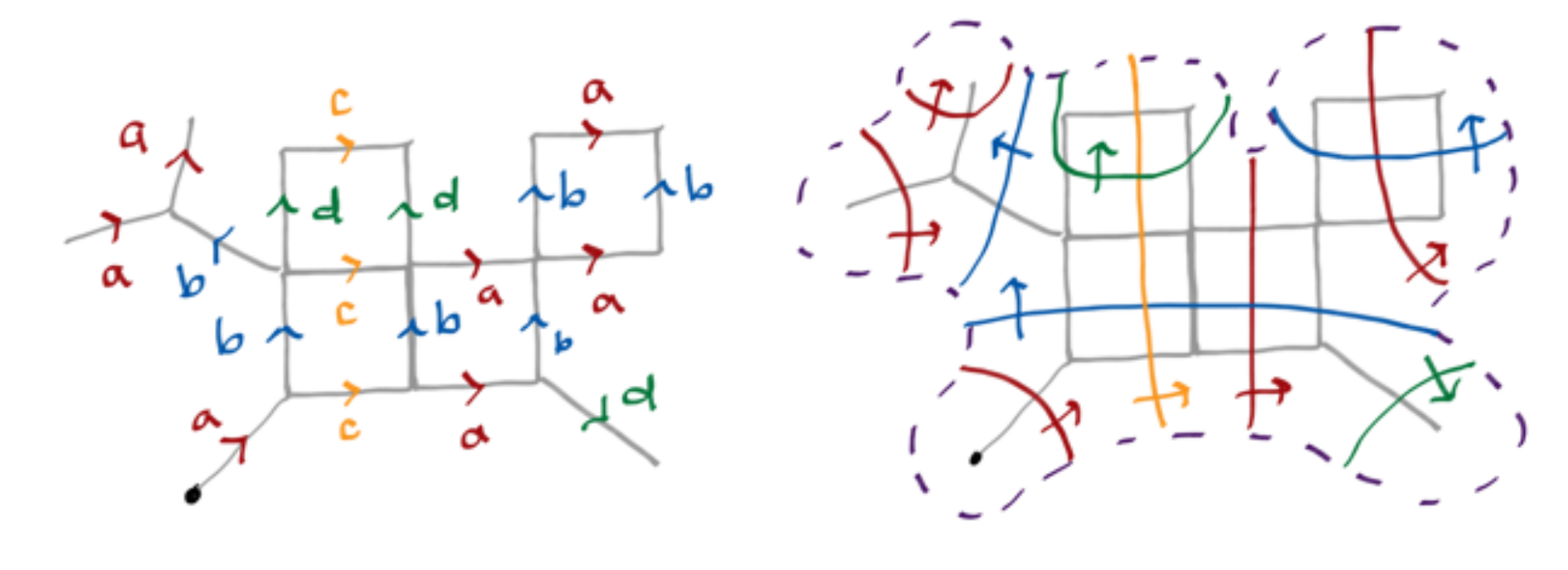}
\caption{On the left, a van Kampen diagram $X$ for the identity word shown in Figure \ref{dd1}, whose disk diagram $\Delta$ is repeated on the right to illustrate their duality.  If the arcs $\mathcal{A}$ of $\Delta$ are extended to a single vertex $v$ in the complement of $\Delta$ in $S^2$, the resulting graph $\mathcal{A}\cup v$ is dual to $X$; compare to Figure 1 in \cite{Kim}.}
\label{dd2}
\end{center}
\end{figure}

While $X$ itself is not necessarily a disk, it has a small neighborhood in $S^2$ which deformation retracts onto $X$, and which is a disk $D$.  The composition of the retract from $D$ to $X$ and the map of $X$ into $\salv$ gives a map $f:D \to \salv$ such that pre-images of hyperplanes are properly embedded arcs in $D$ which induce a disk diagram for $w$.

It is shown in \cite{Kim} that any disk diagram for $w$ is dual in $S^2$ to a van Kampen diagram with boundary word $w$, and vice versa.  In particular, given a disk diagram, one obtains a van Kampen diagram by simply taking its dual complex; see Figure \ref{dd2}.

\subsection{Surgery and subwords}\label{s:surger}
Starting with a disk diagram $\Delta$ for $w$, one can often extract useful information about a subword of $w$ by surgering $\Delta$ along properly embedded arcs. In details, suppose that $\gamma$ is a properly embedded, oriented arc in $\Delta$ which is either an arc of $\mathcal{A}$ or transverse to the arcs of $\mathcal{A}$. Traversing $\gamma$ according to its orientation and recording the labels of those arcs of $\mathcal{A}$ so crossed according to their co-orientation spells a word $y$ in the standard generators. We say the word $y$ is obtained from \emph{label reading} along $\gamma$.

In particular, starting with a subword $w'$  of $w$, any oriented arc $\gamma$ of $D$ which begins at the initial vertex of $w'$ and ends at the terminal vertex of $w'$ produces a word $y$ via label reading such that $w=y$ in $A(\Gamma)$. To see this, note that the arc $\gamma$ cuts the disk $D$ into two disks $D'$ and $D''$, one of which (say $D'$) determines the homotopy (and sequence of moves) to transform the word $w'$ into $y$. This is to say that the disk $D'$ along with arcs coming from $\Delta$ is a disk diagram for the word $w'\bar{y}$, and we say that this diagram is obtained via \emph{surgery} on $\Delta$. This simple observation will be important to our arguments.

Here let us record a few lemmas enabled by surgery.  The first appears as Lemma 2.2 in \cite{Kim}.

\begin{lem}\label{l:starsurgery}Suppose an arc of $\mathcal{A}$ in a disk diagram $\Delta$ for the identity word $w$ cuts off the subword $w'$, i.e., $w \equiv svw'v^{-1}t$, where $s, w',$ and $t$ are subwords and $v,v^{-1}$ are letters at the ends of the arc.  Then $w' \in A(\st(v))$.
\end{lem}

We omit the proof, which is straightforward, and can be found in \cite{Kim}.

If a subword in a disk diagram has the property that no two arcs emanating from it intersect, let us say that this subword is \emph{combed} in the disk diagram.  The property of being combed will be important in Sections \ref{s:ddapps} and \ref{s:jbusting}.

\begin{lem}\label{l:comb}
Suppose $w$ is a word representing the identity and $b$ is a subword of $w$, so $w$ is the concatenation of words $a,b,$ and $c$.  Let $\Delta$ be a disk diagram for $w$.

Then there exists a word $b'$ obtained by re-arranging the letters in $b$, such that $b' = b$ and there exists a disk diagram $\Delta'$ for $ab'c$ in which $b'$ is combed, arcs emanating from $b'$ have the same endpoint in the boundary subword $ca$ as their counterpart in $b$, and arcs that both begin and end in $ca$ are unchanged in $\Delta'$.  

Furthermore, there exists a word $b''$ obtained by deleting letters in $b'$, such that $b''=b$ and there exists a disk diagram $\Delta''$ for $ab''c$ which is precisely $\Delta'$ without the arcs corresponding to the deleted letters.
\end{lem}

\begin{proof}Start with any tight disk diagram $\Delta$ for $w$ and consider the arcs emanating from $b$.  If a pair of arcs from $b$ intersect, they cut out a piece of the disk in which one can find an \emph{innermost} intersection between arcs from $b$, meaning arcs emanating from letters $x$ and $z$ so that $b$ decomposes as $sxyzt$, no arc from a letter in the subword $y$ intersects any other arc, and arcs from $x$ and $z$ intersect at a point $p$.  This means that all arcs that start in $y$ end in $y$.  Consequently, one can divide the disk with an arc $\gamma$ which is disjoint from $\mathcal{A}$ and which divides the boundary into two components, one of which corresponds to $y$.  Surgery using $\gamma$ shows that $y$ equals the identity.  Then $sxyzt = szyxt$.  Let $b_1 \equiv szyxt$ and let $w_1 \equiv ab_1c$.

We now describe a tight disk diagram $\Delta_1$ for $w_1$ in which the number of intersections among arcs emanating from $b_1$ is one less than the number of intersections among arcs emanating from $b$ in the original disk diagram.  In $\Delta$, divide the arcs from $x$ and $z$ into proximal and distal segments, corresponding respectively to the part of the arc between $b$ and the intersection point $p$ and the part of the arc between $p$ and its other endpoint on the boundary.  Because $\Delta$ is tight, an arc intersects the proximal part of the $x$ arc if and only if it intersects the proximal part of the $z$ arc.

Let the arc $z'$ consist of the proximal part of $x$ and the distal part of $z$, and let $x'$ consist of the proximal part of $z$ and the distal part of $x$.  It is evident that $x'$ and $z'$ may be perturbed in a neighborhood of $p$, to be disjoint.  Then $\Delta_1$ is obtained by switching boundary letters $x$ and $z$, and replacing their corresponding arcs with the arcs $x'$ and $z'$.  See Figure \ref{dd3}.  By repeating this process as many times as the original number of intersections between arcs from $b$ in $\Delta$, we obtain a disk diagram $\Delta_n$ for $w_n$ in which the subword $b_n$ is combed, and $b_n$ is equal to $b$ and obtained by rearranging its letters.

To obtain $b''$ and $\Delta''$, delete arcs with both ends in $b'$, by again finding innermost such arcs, which come from subwords of the form $vv^{-1}$. 
\end{proof}

\begin{figure}[H]
\begin{center}
\includegraphics[height=1.5in]{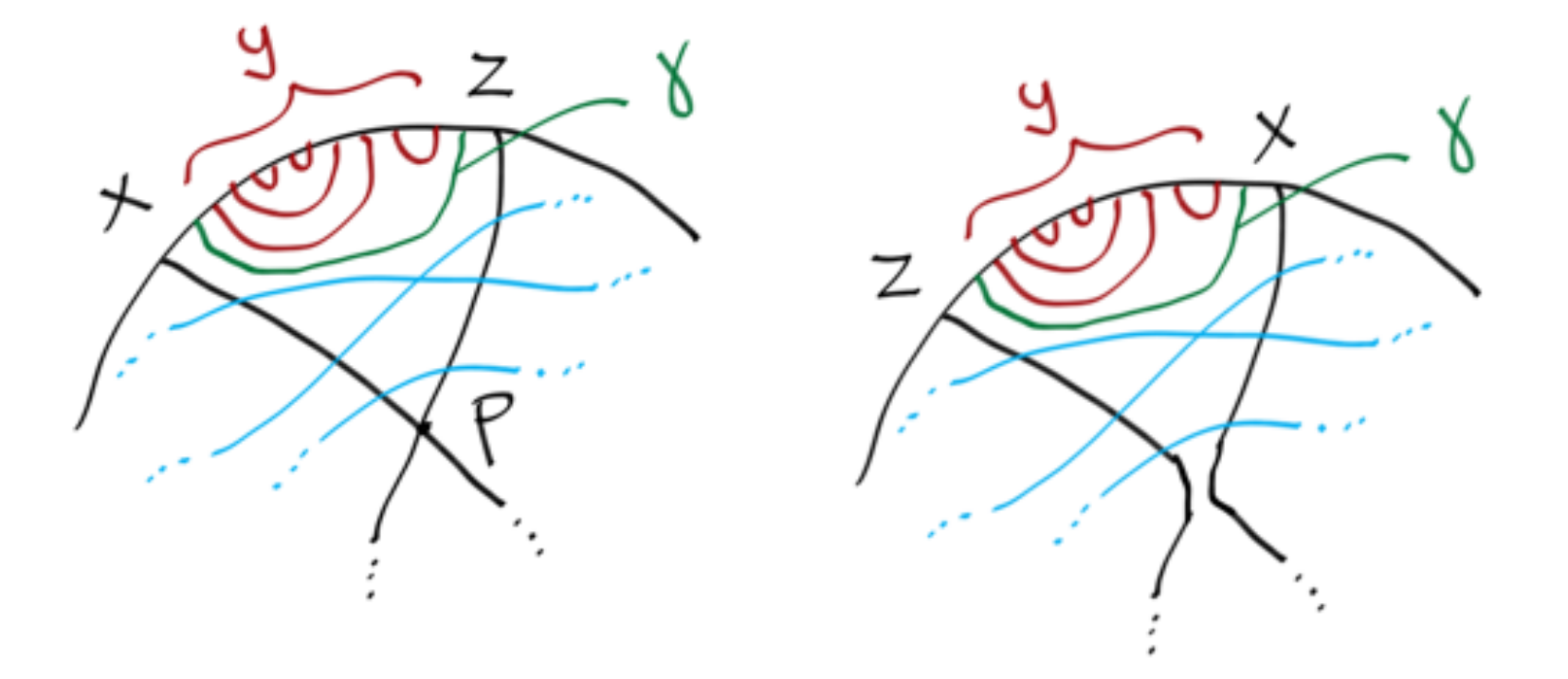}
\caption{On the left, part of $\Delta$; on the right, part of $\Delta_1$.}
\label{dd3}
\end{center}
\end{figure}

\subsection{Reducing diagrams}\label{s:reduce}
In our applications, disk diagrams will come with some additional structure.  Let $h$ be a word in the vertex generators of $A(\Gamma)$, which is not assume to be reduced in any sense. Let $w$ denote a reduced word in the vertex generators which represents the same group element, i.e. $w=h$ in $A(\Gamma)$. Then,  the word $h\bar{w}$ represents the identity in $A(\Gamma)$ and so it is the boundary of some disk diagram $\Delta$. (Here $\bar{w}$ denotes the inverse of the word $w$.) In this situation, we record that the boundary of $\Delta$ consist of two words $h$ and $\bar{w}$ and we draw the diagram as in Figure \ref{rdd}. We sometimes refer to a disk diagram constructed in this way as a \emph{reducing diagram} as it represents a particular way of reducing $h$ to the reduced word $w$. For such disk diagrams, $\partial D$ is divided into two subarcs (each a union of edges) corresponding to the words $h$ and $w$;  we call these subarcs the $h$ and $w$ subarcs, respectively.

\begin{figure}[H]
\begin{center}
\includegraphics{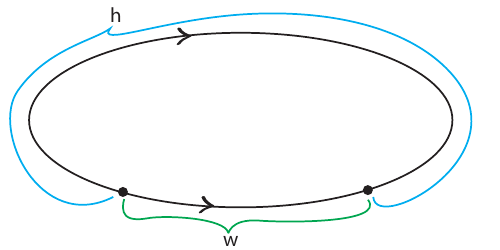}
\caption{Segmented orientation in a reducing diagram.}
\label{rdd}
\end{center}
\end{figure}

Suppose that $\Delta$ is a disk diagram that reduces $h$ to the reduced word $w$. Since $w$ is already a reduced word, no arc of $\mathcal{A}$ can have both its endpoints on the $w$ subarc of $\partial D$.  Otherwise, one could surger the diagram to produce a word equivalent to $w$ with fewer letters. Hence, each arc of $\mathcal{A}$ either has both its endpoints on the subarc of $\partial D$ corresponding to $h$, or it has one endpoint in each subarc of $\partial D$. In the former case, we call the arc (and the letters of $h$ corresponding to its endpoints) \emph{noncontributing} since these letters do not contribute to the reduce word $w$. Otherwise, the arc is called \emph{contributing} (as is the letter of $h$ corresponding the endpoint contained in the $h$ subarc of $\partial D$).  If $h$ as a word is partitioned into a product of subwords $abc$, then the \emph{contribution of the subword $b$ to $w$} is the set of letters in $b$ which are contributing. We remark that whether a letter of $h$ is contributing or not is a property of the fixed disk diagram that reduces $h$ to $w$.

For Section \ref{s:neumann}, we loosen the definition of a reducing diagram so that the word $w$ is only partially reduced.  That is, rather than requiring $w$ be minimal length in $A(\Gamma)$, we simply require that no arc in $\Delta$ has both endpoints in $w$.  Then it is clear that $w$ still has length in $A(\Gamma)$ no longer than that of $h$.

\section{Disk diagram applications}\label{s:ddapps}
From here on, we assume that $\Gamma$ is a connected and anti-connected simplicial graph.
In this section we consider elements of $H$ a finitely generated, star-free subgroup of $A(\Gamma)$; such subgroups include but are more general than purely loxodromic subgroups.  In Section \ref{s:neumann} we include a short proof that star-free subgroups are necessarily free, but we do not require that knowledge here.  We do note that, by finite generation, the inclusion of $H$ into $A(\Gamma)$ is proper with respect to any word metric on both groups. 

Here we employ disk diagrams, described in Section \ref{s:dds}, to derive several useful lemmas about finitely generated, star-free subgroups.  Fix a finite generating set $S$ for $H$ and a reduced word representative for each generator in $S$.  An element $h \in H$ can be expressed as a geodesic word in $H$, that is, $h = h_1\cdots h_n$ such that $h_i \in S$ and $n$ is minimal.  We use a disk diagram with boundary word $(h_1\cdots h_n)h^{-1}$, where $h$ and each $h_i$ are written as reduced words in $A(\Gamma)$.  In other words, we concatenate the reduced word representatives for the $h_i$ to obtain a word representing $h=h_1\cdots h_n$ and consider a reducing diagram for this word. With our choices fixed, we call such a reducing diagram for $h$ simply a disk diagram for $h\in H$. Recall that every letter in the disk diagram is connected by an arc to its inverse somewhere else in the disk diagram, and if arcs intersect, the corresponding generators commute.

\begin{lem}[Cancellation diameter]\label{l:cancel diam}  Suppose $H$ is a finitely generated, star-free subgroup of $A(\Gamma)$.  There exists $D=D(H)$ with the following property:  If in a disk diagram for $h \in H$, a letter in $h_i$ is connected to a letter in $h_j$, then $j-i < D$.
\end{lem}

\begin{proof}
Suppose in a disk diagram for $h \in H$, a letter $g$ in $h_i$ is connected to a letter $g^{-1}$ in $h_j$.  By Lemma \ref{l:starsurgery}, $h_i\cdots h_j = \sigma M\tau$, where $M$ is in the star of $g$, and $\sigma,\tau$ are a prefix of $h_i$ and suffix of $h_j$ respectively.

Therefore, if the lemma is false, there a sequence of reduced-in-$H$ words
\[h^t_{i(t)}\cdots h^t_{j(t)} = \sigma_tM_t\tau_t\]
as above, with $j(t)-i(t)$ increasing.  Because $\Gamma$ is finite and $H$ is finitely generated, we may pass to a subsequence so that the $M_t$ are in the star of the same generator, and furthermore we have constant $\sigma_t = \sigma$ and $\tau_t = \tau$, while $M_t \neq M_s$ for $s \neq t$.  This last claim follows from properness of the inclusion $H \to \aga$. Then for any $s \neq t$,
\[h^s_{i(s)}\cdots h^s_{j(s)}(h^t_{i(t)}\cdots h^t_{j(t)})^{-1} = \sigma M_sM_t^{-1}\sigma^{-1}\]
which is conjugate into a star, a contradiction.
\end{proof}

Consider $h_i \cdots h_j$ a subword of $h = h_1\cdots h_n$ reduced in $H$ as above.  For $D$ the constant of Lemma \ref{l:cancel diam}, we have

\begin{lem}\label{l:bddSWT}  The element $h_i\cdots h_j\in A(\Gamma)$ may be written as a concatenation of three words $\sigma W \tau$, where the letters occurring in $\sigma$ are a subset of the letters occurring in $h_{i-D}\cdots h_{i-1}$ when $i > D$, and in $h_1\cdots h_{i-1}$ otherwise; the letters occurring in $\tau$ are a subset of the letters occurring in $h_{j+1}\cdots h_{j+D}$ when $j \leq n-D$ and in $h_{j+1}\cdots h_n$ otherwise; and the letters occurring in $W$ are exactly the letters occurring in $h_i\cdots h_j$ which survive in the word $h$ after it is reduced in $A(\Gamma)$.
\end{lem}

\begin{proof}
Consider a disk diagram for $(h_1\cdots h_n)h^{-1}$, and let $b$ be the subword $h_i\cdots h_j$.  Lemma \ref{l:comb} provides a word $b'' = b$ and a disk diagram for $(h_1\cdots h_{i-1}b''h_{j+1}\cdots h_n)h^{-1}$ in which all arcs emanating from $b''$ terminate in $h_{j+1}\cdots h_nh^{-1}h_1\cdots h_{i-1}$ and no two arcs from $b''$ intersect.  Then it is clear that $b''$ partitions into subwords $\sigma, W$, and $\tau$, whose arcs connect to $h_1\cdots h_{i-1}$, $h$, and $h_{j+1}\cdots h_n$ respectively.  Recall that the disk diagram furnished by Lemma \ref{l:comb} differs from the original only by removing or shuffling endpoints of arcs in $b$.  Lemma \ref{l:cancel diam} implies that the arcs that do not contribute to $h$ end within the subwords claimed.
\end{proof}

\begin{lem}[Bounded non-contribution]\label{l:contribution} Given $H$ a finitely generated, star-free subgroup of $A(\Gamma)$, there exists $K = K(H)$ such that, if $h_i\cdots h_j$ is a subword of a reduced word for $h$ in $H$ which contributes nothing to the reduced word for $h$ in $A(\Gamma)$, then $j-i \le K$. 
\end{lem}

\begin{proof}
Consider a vanishing subword, that is, a subword that contributes nothing to the reduced word for $h$ in $A(\Gamma)$.  By Lemma \ref{l:bddSWT}, this subword $h_i \cdots h_j$ can be rewritten in $A(\Gamma)$ as $LR$, where the letters occurring in $L$ are a subset of the letters occurring in $h_{i-D}\cdots h_{i-1}$ and the letters occurring in $R$ are a subset of the letters occurring in $h_{j+1}\cdots h_{j+D}$. Therefore $\ell_{A(\Gamma)}(LR) \leq 2D \cdot \max\{\ell_{A(\Gamma)}(h_i)\}$, so that there are only finitely many elements in $A(\Gamma)$ that might equal $LR$. Since the inclusion $H \to A(\Gamma)$ is proper, there is a $K\ge0$ so that no word in $H$ of length greater than $K$ reduces to a work in $A(\Gamma)$ of length less then $2D \cdot \max\{\ell_{A(\Gamma)}(h_i)\}$. Hence, $j - i = |h_i\cdots h_j|_H \le K$.
\end{proof}

\begin{prop}[Undistorted]\label{p:undistorted} Finitely generated star-free subgroups are quasi-isometrically embedded.
\end{prop}

\begin{proof}
For $h \in H < A(\Gamma)$ fulfilling the hypotheses, we have
\[(\ell_H(h) - K)/(K+1) \leq \ell_{A(\Gamma)}(h) \leq \ell_H(h) \cdot \max\{\ell_{A(\Gamma)}(h_i)\}\]
where $K$ is from Lemma \ref{l:contribution}.
\end{proof}

\section{Join-busting in purely loxodromic subgroups}\label{s:jbusting}

Recall that $H$ is \emph{$N$--join-busting} if, for any reduced word $w$ representing $h\in H$, and any join subword $\beta\leq w$, we have $\ell_{\aga}(\beta)\leq N$.  In this section we prove that purely loxodromic subgroups of $A(\Gamma)$ are $N$--join-busting.  Before proving this result (Theorem \ref{t:jbusting} below), let us introduce some terminology. Let $\Lambda$ be a subjoin of $\Gamma$. A reduced word $w$ in generators of $A(\Lambda) < A(\Gamma)$ is said to be $M$--\emph{homogeneous with respect to $\Lambda$} if every subword $w'$ of $w$ with length at least $M$ contains each vertex generator of $\Lambda$, or its inverse. 

\begin{lem}\label{l:homogeneous}
Let $H$ be a finitely generated subgroup of $A(\Gamma)$. Suppose $h_k \in H$ is a sequence of words with reduced decomposition $a_kw_kb_k$, where $w_k$ are join subwords and $\ell_{A(\Gamma)}(w_k) \to \infty$.  Then for some $M \ge 1$ and some subjoin $\Lambda \subset \Gamma$ there exists such a sequence in $H$ so that the $w_k$ are $M$--homogeneous with respect to $\Lambda$.
\end{lem}

\begin{proof}
Since there are only finitely many joins in the graph $\Gamma$, after passing to a subsequence, we may assume that the $w_k$ are supported in a single join $J$.  If there is no $M$ so that the sequence $w_k$ is eventually $M$--homogeneous, then we can find arbitrarily large subwords of $w_k$ supported on a proper subjoin of $J$. After again passing to a subsequence, we may assume that these subjoins are constant, i.e equal to $\Lambda_1 \subset J$. In this manner, we obtain a new sequence of reduced words $h^2_k = a^2_kw^2_kb^2_k$ such that $|w^2_k| \to \infty$ and the letters of $w^2_k$ are contained in $\Lambda_1 \subset J$. Now either $w^2_k$ are eventually $M$--homogeneous with respect to $\Lambda_1$, for some $M \ge 1$, or we may iterate the process and find subjoins \[J \supset \Lambda_1 \supset \Lambda_2 \supset \Lambda_3 \supset \cdots\] in the same way we found $\Lambda_1$. Since there are only finitely many subgraphs of $J$, each of finite size, this process terminates in a sequence of words in $H$ which contribute $M$--homogeneous subwords of unbounded size and which are supported on some subjoin $\Lambda_n \subset J$ of $\Gamma$.
\end{proof}

\begin{thm}\label{t:jbusting}
Let $H<\aga$ be finitely generated and purely loxodromic.  There exists an $N=N(H)$ such that $H$ is $N$--join-busting.
\end{thm}
\begin{proof}
Fix a finite free generating set $S$ for $H$ and reduced word representatives for each generator in $S$.  Suppose the contrary of the theorem, so that we have a sequence $\{h_k\}\subset H$ of elements with reduced decompositions $h_k = a_kw_kb_k$, where the $w_k$ are join subwords of unbounded length.  Fix disk diagrams for $(h_k)^{-1}(a_kw_kb_k)$ so that arcs emanating from $w_k$ do not intersect each other, using Lemma \ref{l:comb}. Recall that we refer to this condition when we say that the arcs for $w_k$ are \emph{combed} in the disk diagram.  By specifying this, we also fix a particular reduced word representing $w_k$. Note that here our disk diagram reduces $h_k$, written as a concatenation of our fixed representatives from $S$, to the reduced word $a_kw_kb_k$.  

By Lemma \ref{l:homogeneous}, we may suppose that the $w_k$ are $M$--homogeneous for some uniform $M$, and that these contributions are supported in a fixed subjoin $\Lambda$.  Recall that this assumption implies that every subword of $w_k$ of length at least $M$ uses each letter coming from the subgraph $\Lambda$.  Because, in the proof of the lemma, our final $w_k$ are subwords of our original $w_k$, they retain the property that arcs emanating from them do not intersect in our fixed disk diagrams (they are combed), and so we can think of the $w_k$ as having fixed reduced form representatives.  We may also assume that $\ell_{\aga}(w_k)\geq kM$, by passing to a subsequence.

Each $h_k$ is represented in its disk diagram by a unique geodesic word $h_{1,k}\cdots h_{s(k),k}$, where the $h_{n,k}$ are from our fixed generating set $S$.  For $k \geq 2$, let $h_{I(k),k}$ be the $H$--generator connected by an arc to the $M^{th}$ letter of $w_k$, and $h_{J(k),k}$ to the $(\ell_{\aga}(w_k)-M+1)^{th}$ letter of $w_k$.  Because the arcs from $w_k$ are combed, $I(k)\leq J(k)$.  For $k$ sufficiently large, depending on how many times the vertices of $J$ correspond to letters in a free generator for $H$, $J(k) - I(k) > 1$.  In that case we may set $i(k) = I(k)+1$ and $j(k) = J(k)-1$ and consider the nonempty subword $h_{i(k),k}\cdots h_{j(k),k}$.

The subword $h'_k = h_{i(k),k}\cdots h_{j(k),k}$ is the largest whose contribution to $w_k$ is \emph{$M$--deep}, meaning that the subword contributes to $w_k$ but none of its arcs connect to either the first or last $M$ letters of $w_k$.  Because the arcs of $w_k$ are combed, all letters of $w_k$ between the first and last sets of $M$ do connect to $h'_k$, and therefore $j(k)-i(k)$ is unbounded.  In particular, for sufficiently large $k$, we may consider the subwords
\[h''_k = h_{i(k)+D,k}\cdots h_{j(k)-D,k}\]
for $D$ the constant from Lemma \ref{l:cancel diam}.

We may apply Lemma \ref{l:bddSWT} and see that $h''_k$ may be written as $\sigma_kW_k\tau_k$, where the letters for $\sigma_k$ correspond to arcs in $h''_k$ that connect to $h_{i(k)}\cdots h_{i(k)+D-1}$, and the letters for $\tau_k$ correspond to arcs in $h''_k$ that connect to $h_{j(k)-D+1}\cdots h_{j(k)}$.  The letters in $W_k$ correspond precisely to the arcs in $h''_k$ that connect to $a_kw_kb_k$.  Any arc in $h''_k$ that connects to either $a_k$ or $b_k$ must cross the arcs from either the first or the last $M$ letters of $w_k$; by our assumption of $M$--homogeneity, this means that those arcs correspond to letters that commute with each vertex generator of  $\Lambda$.  In particular, every letter in $W_k$ either belongs to $\Lambda$ (if it is contributed to $w_k$) or is joined with all of $\Lambda$ (if it is contributed to either $a_k$ or $b_k$).

Now observe that there are only finitely many possibilities for $\sigma_k$ and $\tau_k$, so we may pass to a subsequence where these are constant.  That is, we have a sequence
\[h'''_t = \sigma W_t\tau\]
of words in $H$ where $\ell_H(h'''_t)$ is unbounded and each letter in $W_t$ either belongs to $\Lambda$ or is joined with $\Lambda$.  In particular we have $h'''_s \neq h'''_t$ in this sequence, and \[h'''_t(h'''_s)^{-1} = (\sigma W_t \tau)(\tau^{-1} W_s^{-1} \sigma^{-1})\]
which is conjugate to $W_tW_s^{-1}$.  But $W_tW_s^{-1}$ is a join word, since it is supported on $\Lambda$ joined with $\{\text{letters in }W_t,W_s\text{ not in }\Lambda\}$.  This contradicts the assumption that $H$ is purely loxodromic, by Theorem \ref{thm:loxodromics}.
\end{proof}

\section{Stability of join-busting subgroups}\label{s:stability}
Let $G$ be a finitely generated group and let $H \le G$ be a finitely generated subgroup. We say that $H$ is \emph{stable} in $G$ if $H$ is undistorted in $G$ (i.e. $H \to G$ is a quasi-isometric embedding), and  for any $K\ge1$ there is an $M=M(K)\ge0$ such that any pair of $K$--quasi-geodesics in $G$ with common endpoints in $H$ have Hausdorff distance no greater than $M$.  This definition was introduced by the third author and Durham in \cite{DurhamTaylor}.  It is clear that stable subgroups of $G$ are quasiconvex with respect any word metric on $G$.  In this section we prove that join-busting subgroups of $A(\Gamma)$ are stable.  This will come as a corollary to Theorem \ref{t:stability} below.

Say that an element $g\in A(\Gamma)$ is $N$--join-busting if for any reduced word $w$ representing $g$ and any join subword $w'$ of $w$, $\ell_{A(\Gamma)}(w') \le N$.  Of course, a subgroup in $A(\Gamma)$ is join-busting if and only if its nontrivial elements are $N$--join-busting for some uniform $N$. Recall that the $1$-skeleton of $\Salv$ can be identified with the Cayley graph of $\aga$ with respect to the vertex generators once we fixed a lift of the unique vertex of $S(\Gamma)$. Fix such a lift and denote it by $1$; it is the image of $1\in \aga$ under the orbit map $\aga \to \Salv$ given by $g \mapsto g1$.

\begin{thm} \label{t:stability}
Given $K\ge1$ and $N\ge0$, there exists $M = M(K,N)$ satisfying the following.  Suppose $g \in A(\Gamma)$ is $N$--join-busting and let $w$ be a combinatorial geodesic from $1$ to $g1$. If $p$ is a $K$--quasi-geodesic edge path in $\Salv$ also joining $1$ to $g1$, then the Hausdorff distance between $w$ and $p$ is no more than $M$.
\end{thm}

\begin{proof}

As before, we identify the edge paths $p$ and $w$ with the corresponding sequence of oriented edges, and hence the corresponding word representing $g \in A(\Gamma)$. As $w$ is a geodesic path, it corresponds to a reduced word representing $g$, which we continue to denote by $w$. We wish to bound the distance between the paths $w$ and $p$. To do this, we first fix a disk diagram $\Delta$ realizing the reduction of $p$ to the reduced form $w$. For concreteness, suppose that $\Delta$ is obtained from a map $f: D \to \salv$ as in Section \ref{s:dds}. Retaining the same notation, replace $f$ with the lift $f: D \to \Salv$ such that portion of the boundary of $D$ representing $p$ (or $w$) is mapped to the paths in $\Salv$ spelling $p$ (or $w$). We immediately see that for any two points $x,y \in \partial D$, the combinatorial distance between $f(x)$ and $f(y)$ is no greater than the number of arcs of $\mathcal{A}$ crossed by any path in $D$ from $x$ to $y$.

Set $B = 3N+1$. Let $b$ and $d$ be oriented edges in $w$ that are separated by no fewer that $B-2$ edges of $w$. We first claim that if $\beta$ and $\delta$ are the arcs of $\mathcal{A}$ with endpoints in $b$ and $d$, respectively, then, no arc of $\mathcal{A}$ can intersect both $\beta$ and $\delta$. To see this, let $w' = b \ldots d$  be the length $\ge B$ subpath (subword) of $w$ beginning with the edge $b$ and terminating with the edge $d$. If there is an arc $\epsilon \in \mathcal{A}$ meeting both $\beta$ and $\delta$ then the subdiagram obtained via surgery corresponding to a concatenation of a portion of $\beta$, $\epsilon$, and $\delta$ shows that $w' \in \mathrm{st}(\beta) \cdot \mathrm{st}(\epsilon) \cdot \mathrm{st}(\delta)$. Here, we have continued to use the symbols $\beta$, $\epsilon$, and $\delta$ to denote the vertex generators labeling these arcs. Hence, $w'$ can be written as a product of $3$ join words and since $|w'| \ge B = 3N+1$, one of these join words must have length greater than $N$. (Technically, we have that $w'$ has \emph{some} reduced spelling that is a product of $3$ join words and so one of the join words in this latter spelling has length greater than $N$). This contradicts the assumption that $g$ is $N$--join-busting. Hence, no arc of $\mathcal{A}$ crosses both $\beta$ and $\delta$. In particular, we see that $\beta$ and $\delta$ do not intersect.  

This implies that we can decompose the geodesic $w$ into subpaths as follows:
\[
w= e_1\cdot w_1\cdot e_2\cdot w_2 \cdots e_{n}\cdot w_n,
\]
where $\ell(w_i) = B-1$ for $1\leq i\leq n-1$, $\ell(w_n) \le B-1$, $(n-1)B < l(w) \le nB$, and each $e_i$ is a single edge with the property that if we denote by $\alpha_i$ the arc of $\mathcal{A}$ with endpoint in $e_i$, then no arc of $\mathcal{A}$ crosses any two of the $\alpha_i$. In particular, $\alpha_i$ and $\alpha_j$ are disjoint for $i \neq j$.

Since the $\alpha_i \in \mathcal{A}$ have their other endpoint in edges of the path $p$, we have an induced decomposition:
\[
p= p_0\cdot e_1 \cdot p_1 \cdot e_2 \cdot p_2 \cdots e_{n} \cdot p_n.
\]
Set $l_i = \ell(\alpha_i)$ to be the number of arcs of $\mathcal{A}$ that $\alpha_i$ intersects (i.e. the combinatorial length of the path $\alpha_i$), and set $l_0,l_{n+1} = 0$. We have the following basic fact:
\begin{align}\label{eq:force_up}
\ell(p_k) \ge l_k + l_{k+1} - \ell(w_k),
\end{align}
for all $0\le k\le n$. (Here we have set $w_0$ to be the empty word.) To see this, note that any arc of $\mathcal{A}$ which crosses $\alpha_k$ must terminate at either the path $p_k$ or $w_k$, since the four paths in the above inequality form of subdisk of $\Delta$ and no arcs cross both $\alpha_k$ and $\alpha_{k+1}$. As the same is true for arcs crossing $\alpha_{k+1}$, this proves the inequality.

Now set $M = 2((K+1)(B+1) +K)$ and let $e_{i+1},\ldots,e_{j-1}$ be any maximal consecutive sequence for which $l_{i+1},\ldots,l_{j-1} > M$. For future reference, note that any $e_k$ with $l_k >M$ lives in such a sequence. If no such $e_i$ exist then it is easy to see that the paths $p$ and $w$ have Hausdorff distance bounded in terms of $B, M$, and $K$, and we are done.

For this maximal sequence, consider the subpath of $w$
\[
w' = e_i \cdot w_i \cdot e_{i+1} \cdot w_{i+1} \cdots e_{j-1}\cdot w_{j-1} \cdot e_j
\]
and note that $l_i, l_j \le M$ by maximality. Further, $(j-i-1)B \le \ell(w') \le (j-i)(B+1)$. The subpath $w'$ induces the following subpath of $p$
\[
p' = e_i \cdot p_i \cdot e_{i+1} \cdot p_{i+1} \cdots e_{j-1}\cdot p_{j-1} \cdot e_j,
\]
and using the fact that $p$ is a $K$--quasi-geodesic and $l_i,l_j \le M$, we see that 
\begin{align*} 
\ell(p') &\le K(2M + \ell(w')) +K \\
 &\le K\ell(w') + 2MK +K.
\end{align*}
However, combining Inequality (\ref{eq:force_up}) together with the assumption that $M < l_k$ for $i+1 \le k \le j-1$ gives
\begin{align} \label{e:l}
2M(j-i-1) &\le 2 \sum_{k=i+1}^{j-1} l_k \\
&\le \ell(p') + \ell(w') \\
&\le (K+1)\ell(w') + 2MK +K.
\end{align}

Since $j-i-1 \ge 1$, we have
\begin{align*}
M &\le \frac{K+1}{2} \cdot  \frac{\ell(w')}{j-i-1} + K \cdot  \frac{M}{j-i-1} + \frac{K}{2(j-i-1)}\\
 &\le \frac{K+1}{2} \cdot (B+1)  \frac{j-i}{j-i-1} + K \cdot  \frac{M}{j-i-1} + \frac{K}{2(j-i-1)} \\
 &\le (K+1)(B+1) + K \cdot \frac{M}{j-i-1} + K.
\end{align*}
Since $M = 2((K+1)(B+1) +K)$, the above inequality implies that 
\[
M \le \frac{1}{2}M + K\cdot \frac{M}{j-i-1},
\]
but this is only possible if $j-i \le 2K+1$.
This imposes the uniform bound, $\ell(w') \le (2K+1)(B+1)$. Returning to Inequality (\ref{e:l}) with this new information, we see that 
\begin{align*}
\max_{i\le k \le j} l_k &\le \sum_{k=i+1}^{j-1} l_k \\
&\le \frac{1}{2}(K+1)\ell(w') + MK +\frac{1}{2}K. \\
&\le \frac{1}{2}(K+1)(2K+1)(B+1) + MK +\frac{1}{2}K. 
\end{align*}
If we denote this last quantity by $L$, we see that $L$ depends only on $K$ and $N$. Since any $1\le k \le n$ with $l_k > M$ occurs in such a sequence investigated above, we have $l_k \le L$. Hence, for any $1\le k \le n$, 
\[\ell(p_k) \le K(2L+(B-1)) +K.\]
Hence, any vertex on the path $p$ lies at distance at most \[\frac{1}{2}(K(2L+(B-1)) +K) + L\] from the geodesic $w$. Further, any vertex on $w$ lies at distance at most $\frac{1}{2}(B-1) +L$ from the path $p$. This completes the proof of the theorem by setting 
\[S = \frac{1}{2}(K(2L+(B-1)) +K) + L,\]
where $B = 3N+1$ and \[L =  \frac{1}{2}(K+1)(2K+1)(B+1) + MK +\frac{1}{2}K\] for $M =2((K+1)(B+1) +K)$.
\end{proof}

\begin{cor}[Join-busting implies stable]\label{cor:joinbusting_implies_stable}
Suppose $H \le A(\Gamma)$ is a finitely generated, join-busting subgroup. Then $H$ is stable in $A(\Gamma)$. In particular, $H$ is combinatorially quasiconvex in $A(\Gamma)$.
\end{cor}
\begin{proof}
Suppose that $H$ is $N$--join-busting. By Proposition \ref{p:undistorted}, $H$ is undistorted in $A(\Gamma)$. Hence, to prove stability of $H$ in $A(\Gamma)$ it suffices to show that for any elements $h_1,h_2 \in H$ there is a geodesic $\gamma$ in $A(\Gamma)$ joining $h_1$ to $h_2$ such that if $\alpha$ is a $K$--quasi-geodesic in $A(\Gamma)$ also joining $h_1$ to $h_2$, then the Hausdorff distance between $\alpha$ and $\gamma$ is bounded by $M$, depending only on $K$ and $N$.

To see this, note that $g =h_1^{-1}h_2 \in H$ is $N$--join-busting by definition. By Theorem \ref{t:stability}, if we denote by $w$ a geodesic in $A(\Gamma)$ from $1$ to $g$, then the Hausdorff distance between $w$ and $h_1^{-1} \cdot \alpha$ is bounded by $M$, where $M$ depends only on $N$ and $K$. Here, we are using the identification between the Cayley graph of $A(\Gamma)$ with its vertex generators and $\Salv^{(1)}$. Setting $\gamma = h_1 \cdot w$ completes the proof of the corollary.
\end{proof}

\section{Proof of Theorem \ref{t:main}}\label{s:cnvxccmpct}

In this section we complete the proof of Theorem \ref{t:main}.

\begin{proof}
Suppose $H < A(\Gamma)$ is finitely generated.  First let us prove that (3) implies (1), that is, if $H$ is purely loxodromic then the orbit map of $H$ into $\Gamma^e$ is a quasi-isometric embedding.  By Theorem \ref{t:jbusting}, we know that, for any element of $H$ written as a geodesic word in $A(\Gamma)$, the length of any join subword is bounded.  As star subwords are join subwords, this implies that star-length is bi-Lipschitz to $A(\Gamma)$--length for elements in $H$.  On the other hand, orbit distance in $\Gamma^e$ is quasi-isometric to star-length \cite{KimKoberda2014}.   Moreover, Proposition \ref{p:undistorted} says that $|\cdot|_H$ is quasi-isometric to $|\cdot|_{A(\Gamma)}$ for elements in $H$.  Therefore we may conclude that, for $H$ a purely loxodromic, finitely generated subgroup of $A(\Gamma)$, the orbit map of $(H, |\cdot|_H)$ into $\Gamma^e$ is a quasi-isometric embedding.  The converse is immediate: if $H \to \Gamma^e$ is a quasi-isometric embedding, no element of $H$ can be elliptic, because these have bounded orbit, and right-angled Artin groups are torsion-free.

We have seen in Corollary \ref{cor:joinbusting_implies_stable} that join-busting implies stability; this and Theorem \ref{t:jbusting} proves that (3) implies (2).  The converse follows from the observation that, if a subgroup is stable (and thus hyperbolic), then all its $\Z$--subgroups are Morse and therefore stable (compare with Chapter III.$\Gamma$, Proposition 3.9 and Corollary 3.10 of~\cite{bridhaef}), and is therefore generated by loxodromic elements.  To see the lattermost claim, consider the cyclic subgroup generated by an elliptic element $w$, whose centralizer contains some nontrivial $c$ with $\langle c, w \rangle \cong \Z \oplus \Z$.  If we represent $w$ and $c$ by reduced words, then the edge paths corresponding to $c^nw^nc^{-n}$ and $w^n$ are pairs of $K$--quasi-geodesics whose Hausdorff distance grows with $n$, with $K$ independent of $n$. To see this last claim, we have that $\langle c,w\rangle$ quasi-isometrically embeds in $A(\Gamma)$, as follows from the Flat Torus Theorem (see II.7 of \cite{bridhaef}). It therefore suffices to see that the edge paths corresponding to $c^nw^nc^{-n}$ and $w^n$ are pairs of uniform quasi-geodesics in $\Z^2$. This is a straightforward computation. Therefore $\grp{w}$ is not stable.
\end{proof}

\section{The Howson property for star-free subgroups}\label{s:neumann}

In this section we complete the proof of Theorem \ref{t:starfree}.  Let $v\in V(\gam)$, and let \[G_v=\langle \{w\mid w\in V(\st(v))\}\rangle,\] i.e. the subgroup of $\aga$ generated by the vertices in the star of $v$.  We say that such a group $G_v$ is a \emph{star subgroup} of $\aga$.  Let $H<\aga$ be a subgroup.  We say that $H$ is \emph{star-free} if for every $1\neq g\in H$, we have that $g$ is not conjugate into a star subgroup of $\aga$.  In other words, there is no choice $h\in\aga$ and $v\in V(\gam)$ such that $g^h\in G_v$.  In the language of \cite{KimKoberda2014}, $H$ is star-free if for every $1\neq g\in H$, a cyclically reduced conjugate of $g$ always has star length at least two. Here, the \emph{star length} of an element $g\in \aga$ is the smallest $N$ for which there are elements $\{h_1,\ldots,h_N\}$ such that each $h_i$ is contained in a star subgroup of $\aga$, and such that $g=h_1\cdots h_N$.

Note that a purely loxodromic subgroup of $\aga$ is star-free, but the converse is false.  For example, one can take \[F_2\times F_2\cong \langle a,c\rangle\times\langle b,d\rangle.\]  If $\gam$ is a square labeled cyclically by the vertices $\{a,b,c,d\}$, we have that \[\aga\cong  \langle a,c\rangle\times\langle b,d\rangle.\]  Note that $\aga$ has no loxodromic elements in this case, but any cyclically reduced word with full support is star-free.  The subgroup $\langle abcd\rangle$ is cyclic, and every element is star-free. The following proposition gives the exact relationship between star-free and purely loxodromic subgroups of $A(\Gamma)$:

\begin{prop}\label{p:starfree_vs_plox}
Let $H\le \aga$ be a finitely generated subgroup. Then $H$ is purely loxodromic if and only if $H$ is star-free and combinatorially quasiconvex.
\end{prop}
\begin{proof}
If $H$ is purely loxodromic, then it is star-free by definition and quasiconvex (with respect to any generating set) by Theorem \ref{t:main}.

Now suppose that $H$ is star-free and combinatorially quasiconvex but not purely loxodromic. Clearly, being star-free is invariant under conjugation. Combinatorial quasiconvexity is also invariant under conjugation for right-angled Artin groups, though this is less obvious. The reader may consult Corollary 2.29 of \cite{haglund2008}, where it is shown that if $H$ is combinatorially quasiconvex, then so is the coset $H \cdot g \subset A(\Gamma)$. This in turn implies that $g^{-1}Hg$ is a combinatorially quasiconvex subgroup of $A(\Gamma)$.

We may therefore suppose that there is a join $J \subset \Gamma$ such that \[H_J =H \cap A(J) \neq \grp{1}.\] Since both $H$ and $A(J)$ are combinatorially quasiconvex in $A(\Gamma)$, $H_J$ is combinatorially quasiconvex in $A(J)$. This last claim follows from the proof of Proposition III.$\Gamma$.4.13 in \cite{bridhaef}, where the reader may note that the semihyperbolicity assumption is never used. Since $J$ is a join, $J = X *Y$ where $X$ and $Y$ are nonempty subgraphs of $\Gamma$ each of which is contained in some star of $\Gamma$. This implies that $H_J \le A(J) = A(X) \times A(Y)$ and 
\[H_J \cap A(X) = H_J \cap A(Y)=\grp{1},\]
because $H$ is star-free.

Let $K\ge 0$ be the quasiconvexity constant for $H_J \le A(J) = A(X) \times A(Y)$. If $h_i \in H_J$ are distinct elements whose word lengths go to infinity, then $h_i=x_iy_i$ and the lengths of either $x_i \in A(X)$ or $y_i \in A(Y)$ also go to infinity. Suppose $|x_i|_{A(J)} \to \infty$. Then since the concatenation $[1,x_i] \cup [x_i,x_iy_i]$ is a geodesic in $A(J) = A(X) \times A(Y)$ and $H_J$ is $K$-quasiconvex, there must be a sequence $p_i \in A(X)$ and $q_i \in A(Y)$ with $p_iq_i \in H_J$, $p_iq_i$ is within $K$ of the point $x_i$, and $|q_i|_{A(Y)} \le K$. As there are only finitely many elements of $A(Y)$ of length no greater than $K$, there are $i\neq j$ with $q_i=q_j$. Then 
\[1\neq p_iq_i(p_jq_j)^{-1} = p_ip_j^{-1} \in H_J \cap A(X),\]
a contradiction. 
\end{proof}

Star-free subgroups are always free, even without the assumption of finite generation:

\begin{prop}\label{p:free}
Let $H<\aga$ be a star-free subgroup.  Then $H$ is free.
\end{prop}
In particular, we obtain part (1) of Theorem \ref{t:starfree}.  The proof of this proposition is identical to the proof of Theorem 53 in \cite{KimKoberda2014} [cf. Theorem \ref{t:gex}, item (7)].  We recall a proof for the convenience of the reader.
\begin{proof}[Proof of Proposition \ref{p:free}]
Suppose $H \neq 1$. Choose a vertex $v\in V(\gam)$, and let $\gam_v$ be the subgraph of $\gam$ spanned by $V(\gam)\setminus \{v\}$.  Then we have \[\aga\cong \langle \st(v)\rangle*_{\langle\lk(v)\rangle}A(\gam_v),\] so we get an action of $\aga$ on the tree associated to this splitting.  Since $H$ is star-free, the action of $H$ has trivial edge stabilizer, so that $H$ splits as a free product of subgroups of $A(\gam_v)$ with possibly an additional free factor.  The proposition follows by induction on $|V(\gam)|$.
\end{proof}

Henceforth in this section, we will fix a finite free generating set $S$ for $H$ as a subgroup of $\aga$. We also fix reduced word representatives for each element of $S$.

Note that Proposition \ref{p:undistorted} gives part (2) of Theorem \ref{t:starfree}.  So, we are left with establishing part (3).  In the remainder of this section, we will prove:

\begin{thm}\label{t:hannaneumann}
Let $\gam$ be a finite simplicial graph, $\Lambda\subset\gam$ an induced subgraph, and let $H<\aga$ be a finitely generated star-free subgroup.  Then the group $H_{\Lambda}=H\cap A(\Lambda)$ is also finitely generated.
\end{thm}

Note that Theorem \ref{t:hannaneumann} does not exclude the possibility that $H_{\Lambda}$ is the identity, and this will indeed be the case if $\Lambda$ is contained in the star of a vertex of $\Gamma$.  Theorem \ref{t:hannaneumann} is false without the assumption that $H$ is star-free, as the following example which the authors learned from M. Sapir will demonstrate.  Consider the group $F_2\times\bZ$, where the factors are generated by $\{a,b\}$ and $z$ respectively.  This group is $A(P_3)$, where $P_3$ is the path on three vertices, with middle vertex $z$ and end vertices $a$ and $b$.  Consider the subgroup $H=\langle az,b\rangle$, which is a free subgroup of $A(P_3)$, and let us look at its intersection with $\langle a,b\rangle$.  We have that a word $w$ in $az$ and $b$ lies in $\langle a,b\rangle$ if and only if the $z$ exponent of $w$ is zero, i.e. $w$ is in the kernel of the homomorphism $H\to \bZ$ defined by $az\to 1$ and $b\to 0$.  It follows that $H\cap\langle a,b\rangle$ is infinitely generated.

In the proof of Theorem \ref{t:hannaneumann}, we use the fact that we can control cancellation distance within star-free subgroups in order to produce an explicit finite generating set for $H_{\Lambda}$.

\subsection{Cancellation patterns}

Before proceeding, we must define cancellation patterns of vertices within words in the star-free subgroup $H$.  We will use the disk diagram terminology outlined in Section \ref{s:dds}.  Let $h\in H$ be an element viewed as a word in the free generating set $S$ for $H$, and fix a reducing diagram $\Delta$ for $h$ which realizes $h$ as the reduced word $w$ in $\aga$.  Using Lemma \ref{l:comb}, we will assume the the disk diagram $\Delta$ and the reduced word $w$ have been chosen so that $w$ is combed in $\Delta$.
Thus, each letter of $h$ has a co-oriented arc in $\Delta$ attached to it, which either has both endpoints in the subarc of $\partial\Delta$ labelled by $h$ (in which case the two vertices cancel), or which has one endpoint in the subarc of $\partial\Delta$ labelled by $w$ (in which case the vertex does not cancel and in which case we call the arc \emph{contributing}).  Arcs in the reducing diagram which are not contributing will be called \emph{non-contributing} arcs.  Every arc in $\Delta$ is labelled by a vertex of $\gam$ or its inverse, depending on the co-orientation of the arc.

\begin{figure}[H]
\begin{center}
\includegraphics[height=1.50in]{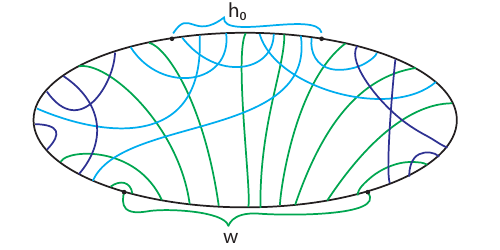}
\caption{Reducing diagram indicating subword $h_0$.  Light blue and dark blue arcs have both ends in $h$.}
\label{subword}
\end{center}
\end{figure}

Let $h_0$ be an \emph{$H$--subword} of $h$, which is to say that if \[h=s_1\cdots s_n,\] where each $s_i\in S$, we have that \[h_0=s_i\cdots s_j\] for some $1\leq i\leq j\leq n$.  We will often use the notation $h_0\leq h$ for a subword of $h$.

We now fix the $H$--subword $h_0\leq h$ and a reducing diagram $\Delta$ for $h$.  Suppose $v\in V(\gam)$.  A \emph{$v$--cancellation partition for $h_0$} is a partition of the occurrences of $v$ and $v^{-1}$ in $h_0$, viewed as an unreduced word in the vertex generators of $\aga$, into eight sets \[\{L_v^{\pm},R_v^{\pm},I_v^{\pm},N_v^{\pm}\},\] where a set with a positive exponent only contains copies of the element $v$ and a set with a negative exponent only contains copies of the element $v^{-1}$.  Intuitively, $h_0\leq h\in H$ is a subword in which occurrences of $v$ cancel in the reducing diagram $\Delta$.  The sets $L_v^+$ and $R_v^+$ correspond to the copies of $v$ which cancel to the left of $h_0$ and to the right if $h_0$, while the sets $L_v^{-}$ and $R_v^{-}$ play the same respective roles for copies of $v^{-1}$.  The sets $I_v^{\pm}$ correspond to copies of $v$ and $v^{-1}$ which cancel internally within $h_0$, and $N_v^{\pm}$ corresponds to copies of $v$ and $v^{-1}$ which do not cancel at all when $h$ is reduced within $\aga$.  We will say that two words have the \emph{same cancellation partition} if the corresponding sets \[\{L_v^{\pm},R_v^{\pm},I_v^{\pm},N_v^{\pm}\}\] have the same cardinality.  If the word $h_0$ is not clear from context, we will write \[\{L_v^{\pm}(h_0),R_v^{\pm}(h_0),I_v^{\pm}(h_0),N_v^{\pm}(h_0)\}\] to eliminate any ambiguities.

In a reducing diagram $\Delta$ for $h$, the copies of $v$ and $v^{-1}$ in $L_v^+$ and $L_v^-$ have non-contributing arcs emanating from them that move off to the left of $h_0$.  Similarly, the copies of $v$ and $v^{-1}$ in $R_v^+$ and $R_v^-$ have arcs emanating from them that move off to the right of $h_0$.  The copies of $v$ and $v^{-1}$ in $I_v^+\cup I_v^-$ are connected by arcs with both endpoints in $h_0$, and the copies of $v$ and $v^{-1}$ in $N_v^+\cup N_v^-$ have vertical arcs connecting them to $w$.

Observe that group inversion relates a cancellation partition for a subword $h_0\leq h$ with a cancellation partition for the subword $h_0^{-1}\leq h^{-1}$, with respect to the reducing diagram for $h^{-1}$ induced by the inversely labeled mirror image of the reducing diagram for $h$.  Specifically, we have \[I_v^{\pm}(h_0)=I_v^{\mp}(h_0^{-1}), N_v^{\pm}(h_0)=N_v^{\mp}(h_0^{-1}),\] and similarly \[R_v^{\pm}(h_0)=L_v^{\mp}(h_0^{-1}),L_v^{\pm}(h_0)=R_v^{\mp}(h_0^{-1}).\]

We are now ready to define the cancellation pattern for a subword $h_0\leq h$.  We delete all the arcs in the diagram $\Delta$ which do not have an endpoint in $h_0$, and for the remaining arcs, we retain the original co-orientation.  Furthermore, we label each arc by the vertex of its terminal endpoint in $h_0$.  We then draw two unoriented, simple arcs $\iota$ and $\tau$ (for ``initial'' and ``terminal'') in  $\Delta$ which satisfy the following conditions:

\begin{enumerate}
\item The arcs $\iota$ and $\tau$ connect the subarc of $\partial\Delta$ labelled by $h$ to the subarc of $\partial\Delta$ labelled by $w$.

 \item Neither $\iota$ nor $\tau$ intersect any arcs of $h$ that contribute to the reduced word $w$. (This is possible since the word $w$ is combed in $\Delta$.)
\item
If a non-contributing arc has exactly one endpoint in $h_0$, then the cancellation arc intersects exactly one of $\iota$ and $\tau$ exactly once.
\item
If $\al$ and $\al'$ are non--contributing arcs that both intersect either $\iota$ or $\tau$, then $\al$ and $\al'$ do not intersect in the region between $\iota$ and $\tau$.
\end{enumerate}

The authors are grateful to the referee for pointing out that with this definition, a cancellation pattern is a well--defined invariant of a subword $h_0$ of $h$.

\begin{figure}[H]
\begin{center}
\includegraphics[height=1.50in]{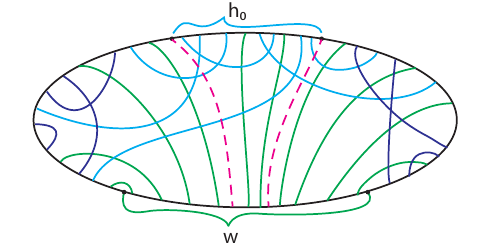}
\caption{Choice of $\iota$ and $\tau$.}
\label{interarcs}
\end{center}
\end{figure}

The \emph{cancellation pattern} $\mathcal{P}(h_0)$ of $h_0\leq h$ with respect to the chosen reducing diagram $\Delta$ is the closed region of the reducing diagram bounded by the boundary arcs, $\iota$, and $\tau$, together with the co-oriented, labelled non-contributing arcs. The boundary of $\mathcal{P}(h_0)$ consists of four arcs, namely $\iota$, $\tau$, and the \emph{top and bottom boundary arcs}. The top and bottom boundary arcs are labeled by the letters which occur in $h_0$ and $w$ respectively. Precisely, the intersection point of a contributing or non--contributing arc and a top or bottom boundary arc is labeled by the corresponding letter.

Two words $h_1$ and $h_2$ have the \emph{same cancellation pattern} if $\mathcal{P}(h_1)$ is homeomorphic to $\mathcal{P}(h_2)$ via a homeomorphism which is a label and co-orientation-preserving homeomorphism of the union of the non-contributing arcs, and which preserves the boundary arcs of the cancellation pattern.  Informally, two words have the same cancellation pattern if the cancellations of letters in the two words happen ``in the same way.''  It is clear from the pigeonhole principle that a fixed word $h_0\leq h$ in $H$ can have only finitely many different cancellation patterns as $h$ varies among all possible words in $H$ which contain $h_0$ as a subword, up to the equivalence relation of having the same cancellation pattern.

Fixing a reducing diagram $\Delta$ for $h$ and a subword $h_0\leq h$, a cancellation pattern of $h_0$ is related to a cancellation pattern for $h_0^{-1}\leq h^{-1}$ in a manner analogous to that of cancellation partitions.  Again, we can build a reducing diagram for $h^{-1}$ by taking the mirror image of the reducing diagram for $h$, inverting all the labels on the top and bottom boundary arcs, and switching all the co-orientations on the  non-contributing arcs.  Thus, the natural choice for the cancellation pattern $\mathcal{P}(h_0^{-1})$ in $h^{-1}$ is just the mirror image of $\mathcal{P}(h_0)$, with the labels $\iota$ and $\tau$ switched, the labels on boundary arcs inverted, and the co-orientations of the non-contributing arcs switched.

\begin{figure}[H]
\begin{center}
\includegraphics[height=1.50in]{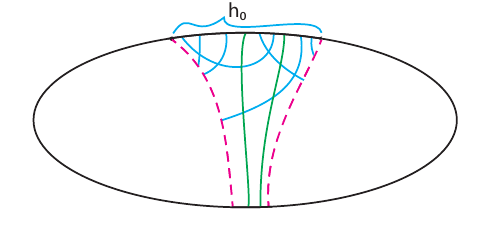}
\caption{Cancellation pattern $\mathcal{P}(h_0)$.}
\label{canpat}
\end{center}
\end{figure}

We remark briefly that when we define a cancellation pattern for a subword $h_0\leq h$, we keep track of all the non-contributing arcs, not just the ones corresponding to a particular vertex of $\gam$.  This point is essential in the proof of Theorem \ref{t:hannaneumann}.

\subsection{The Howson Property}
\begin{proof}[Proof of Theorem \ref{t:hannaneumann}]
By induction on $|V(\gam)|-|V(\Lambda)|$, it suffices to prove the result for $\Lambda$ equal to $\gam$ with one vertex, say $v$, deleted.  Given $h\in H_\Lambda = H\cap A(\Lambda)$, we will decompose $h$ into a product of two words $h=h_1\cdot h_2$, where $h_1,h_2\in H_\Lambda$, where $h_1$ has $H$--length bounded by $K=K(H)$, and where $h_2$ has $H$--length shorter than the maximum of $K$ and the $H$--length of $h$.  Here, the $H$--length of an element of $H$ is the word length with respect to the free generating set $S$.  This will prove the result, since it will show that $h$ is a product of elements of $H_\Lambda$ of bounded $H$--length.

Observe that the statement $h\in H_{\Lambda}$ is just the statement that in some (and hence in every) reducing diagram for $h$, every occurrence of $v$ and $v^{-1}$ cancels.  In other words, we have \[N_v^+(h_0)\cup N_v^-(h_0)=\emptyset\] in the cancellation partition for every subword $h_0$ of $h$.

Writing \[h=s_1\cdots s_n,\] Lemma \ref{l:cancel diam} says that if $v\in\supp s_i$ cancels with $v^{-1}\in\supp s_j$, then $|j-i|\leq D$, where $D$ is a constant which depends only on $H$.

We consider the set $\{\al_i\}$ of initial $H$--segments of $h$, which is to say that if \[h=s_1\cdots s_n,\] then \[\al_i=s_1\cdots s_i.\]  Given an initial segment $\al_i$, we call $\omega_i=\al_i^{-1}h$ the terminal $H$--segment of $h$.  For notational convenience, we will assume that $D\leq i\leq n-D$.  Since $D$ is fixed and depends only on $H$, this assumption is valid, since there are only finitely many elements of $H$ of $H$--length bounded by $D$.

For each $i$, we write $b_i$ and $e_i$ for the terminal $H$--segment of $\al_i$ of $H$--length $D$, and the initial $H$--segment of $\omega_i$ of $H$--length $D$, respectively.  That is to say, we have \[b_i=s_{i-D+1}\cdots s_i,\] and \[e_i=s_{i+1}\cdots s_{i+D}.\]  Observe that if $n\gg 0$ then there are two indices, say $i$ and $j$ with $i<j$, for which the segments $b_i$ and $b_j$ are equal as words in $H$, and for which the cancellation patterns are the same.

\begin{figure}[H]
\begin{center}
\includegraphics[height=2in]{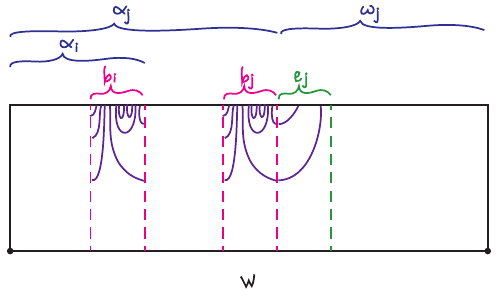}
\caption{Part of reducing diagram with $b_i = b_j$ and $e_j$ indicated; only arcs corresponding to $v$ are shown, although other arcs appear in their cancellation patterns.  However, by Lemma \ref{l:cancel diam}, no arcs traverse these three cancellation patterns, since they have $H$-length $D$.}
\label{HNarg}
\end{center}
\end{figure}

Since we may assume $j$ and $i$ to be bounded in a way which depends only on $H$ (since we produce $i$ and $j$ via a pigeonhole principle argument), we have that $\al_j\cdot \al_i^{-1}$ has universally bounded $H$--length, say $K=K(H)$.  Observe that both $\al_j\cdot \al_i^{-1}$ and $\al_i\cdot\omega_j$ have $H$--length shorter than the maximum of that of $h$ and $K$.  Moreover, we claim that $\al_j\cdot \al_i^{-1}$ and $\al_i\cdot\omega_j$ both lie in $H_{\Lambda}$.   This last claim will suffice to establish the result.

We consider $\al_j\cdot \al_i^{-1}$ first.  Observe that we may write $\al_j=\beta_j\cdot b_i$ and $\al_i=\beta_i\cdot b_i$ as reduced words in $H$, since $b_i=b_j$.  By Lemma \ref{l:cancel diam}, in any reducing diagram for $h$, occurrences of $v$ or $v^{-1}$ in $\al_i$ that cancel to the right of $\al_i$ must lie in $b_i$; likewise for $\al_j$ and $b_j$.  Furthermore, since $\al_i$ and $\al_j$ are initial segments of $h\in H_{\Lambda}$, no occurrence of $v$ or $v^{-1}$ cancels to the left.  Thus, the only $v$--cancellation arcs for $\al_i$ which do not have both endpoints in $\al_i$ are precisely the ones coming from the sets $R_v^{\pm}(b_i)$, and likewise for $\al_j$ and $R_v^{\pm}(b_j)$.

We may now construct a reducing diagram for \[\beta_jb_i\cdot b_i^{-1}\beta_i^{-1}=\al_j\cdot\al_i^{-1}\] by placing the cancellation patterns $\mathcal{P}(b_i)$ and $\mathcal{P}(b_i^{-1})$ next to each other, pairing the $\tau$--arc of $\mathcal{P}(b_i)$ with the $\iota$--arc of $\mathcal{P}(b_i^{-1})$.  We then splice the cancellation arcs together by matching arcs with compatible labels and co-orientations.  By the choice of $D$, no non-contributing arc with an endpoint in $\beta_j$ traverses $\mathcal{P}(b_i)$, and no non-contributing arc with an endpoint in $\beta_i^{-1}$ traverses $\mathcal{P}(b_i^{-1})$.  The result is that every occurrence of $v$ and $v^{-1}$ in $\al_j\cdot\al_i^{-1}$ cancels in $A(\Lambda)$, which is to say that $\al_j\cdot\al_i^{-1}\in H_{\Lambda}$.

\begin{figure}[H]
\begin{center}
\includegraphics[height=2in]{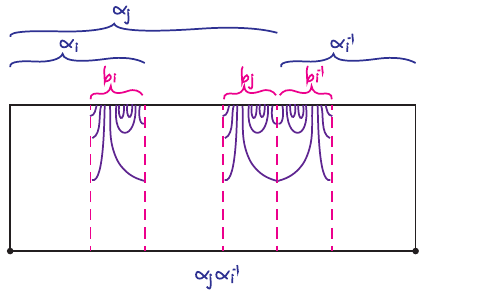}
\caption{Part of reducing diagram for $\al_j\cdot \al_i^{-1}$.}
\label{HNarg}
\end{center}
\end{figure}

The case of $\al_i\cdot\omega_j$ is similar but simpler.  Consider the subword $b_i\cdot e_j$ of $\al_i\cdot\omega_j$.  Since $b_i=b_j$ as words in $H$ and since the cancellation patterns of $b_i$ and $b_j$ are the same as subwords of $h$, we have that every $v\in R_v^+(b_i)$ can be cancelled with a $v^{-1}\in L_v^-(e_j)$, and similarly every $v^{-1}\in R_v^-(b_i)$ can be cancelled with a $v\in L_v^+(e_j)$.  Observe that every copy of $v$ and $v^{-1}$ which does not lie in $R_v^{\pm}(b_i)$ lies in $I_v^{\pm}(\al_i)$.  Similarly, every copy of $v$ and $v^{-1}$ which does not lie in $L_v^{\pm}(e_j)$ lies in $I_v^{\pm}(\omega_j)$.  Thus, every occurrence of $v$ and $v^{-1}$ cancels in $\al_i\cdot\omega_j$, so that this element also lies in $H_{\Lambda}$.
\end{proof}

\begin{figure}[H]
\begin{center}
\includegraphics[height=2in]{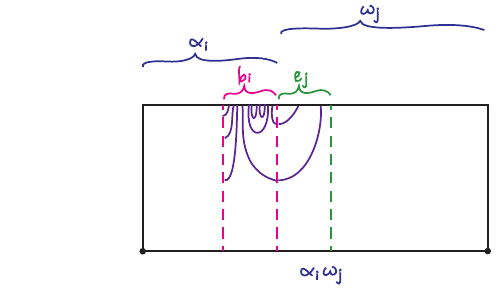}
\caption{Part of reducing diagram for $\al_i\cdot\omega_j$.}
\label{HNarg}
\end{center}
\end{figure}

\bibliography{biblio}
\bibliographystyle{amsalpha}

\end{document}